\documentclass[a4paper, 12pt]{amsart}
\usepackage{amssymb, amsmath, amsthm, enumerate, color, subcaption, tipa, upgreek, xcolor}
\usepackage{amstext, tikz, appendix, mathrsfs, mathabx, mathtools, pgfplots, stmaryrd}
\usepackage{csquotes}
\usepackage[colorlinks=true, breaklinks=true, linkcolor=black, citecolor=blue, urlcolor=red]{hyperref} 
\usepackage[english]{babel}
\usepackage{forest, adjustbox}
\usepackage{tikz-cd}
\usepackage{enumitem}

\numberwithin{equation}{section}
\newtheorem{letterthm}{Theorem}

\newtheorem{lettercor}[letterthm]{Corollary}

\newtheorem{theorem}{Theorem}[section]
\newtheorem{lemma}[theorem]{Lemma}
\newtheorem{corollary}[theorem]{Corollary}
\newtheorem{proposition}[theorem]{Proposition}
\newtheorem{observation}[theorem]{Observation}

\newtheorem{definition}[theorem]{Definition}

\newtheorem{remark}[theorem]{Remark}
\newtheorem{example}[theorem]{Example}

\newcommand{\act}{\curvearrowright}

\newcommand{\cC}{\mathcal C}
\newcommand{\C}{\mathbf C}

\DeclareMathOperator{\Caret}{Caret}

\DeclareMathOperator{\Comp}{Comp}

\DeclareMathOperator{\End}{End}
\newcommand{\cF}{\mathcal F}

\DeclareMathOperator{\Func}{Func}

\newcommand{\Ga}{\Gamma}
\DeclareMathOperator{\Hom}{Hom}
\DeclareMathOperator{\id}{id}
\DeclareMathOperator{\Irr}{Irr}
\newcommand{\into}{\hookrightarrow}

\newcommand{\bk}{\mathbf k}

\DeclareMathOperator{\Leaf}{Leaf}

\newcommand{\N}{\mathbf{N}}
\DeclareMathOperator{\nd}{{nd}}
\newcommand{\cO}{\mathcal O}
\newcommand{\ot}{\otimes}

\newcommand{\cP}{\mathcal P}

\newcommand{\R}{\mathbf{R}}

\DeclareMathOperator{\Rep}{Rep}

\DeclareMathOperator{\full}{{full}}

\newcommand{\cS}{\mathcal S}

\newcommand{\cT}{\mathcal T}

\newcommand{\ti}{\tilde}

\newcommand{\cV}{\mathcal V}

\DeclareMathOperator{\Vect}{Vec}
\newcommand{\Y}{\wedge}
\newcommand{\Z}{\mathbf{Z}}
\newcommand{\onto}{\twoheadrightarrow}
\DeclarePairedDelimiterX{\norm}[1]{\lVert}{\rVert}{#1}

\begin{document}

	\title{Moduli of representations of Leavitt path algebras}
	\thanks{
		AB is supported by the Australian Research Council Grant DP200100067.\\
		DW is supported by an Australian Government Research Training Program (RTP) Scholarship.}
	\author{Arnaud Brothier and Dilshan Wijesena}
\address{Arnaud Brothier\\	University of Trieste, Department of Mathematics, via Valerio 12/1, 34127, Trieste, Italy and School of Mathematics and Statistics, University of New South Wales, Sydney NSW 2052, Australia}
\email{arnaud.brothier@gmail.com\endgraf
		\url{https://sites.google.com/site/arnaudbrothier/}}
	\address{Dilshan Wijesena\\ School of Mathematics and Statistics, University of New South Wales, Sydney NSW 2052, Australia}
	\email{dilshan.wijesena@hotmail.com}
	\maketitle
	

\begin{abstract}
We transpose Jones' technology and the authors' $C^*$-algebraic techniques to study representations of the Leavitt path algebra $L$ (over an arbitrary row-finite graph) by using its quiver algebra $A$.
We establish an equivalence of categories between certain full subcategories of $\Rep(A)$ and $\Rep(L)$ that preserves irreducibility and indecomposability.
We define a dimension function on $\Rep(L)$, and for each finite dimension we provide a moduli space for the irreducible classes by transporting structures of King and Nakajima on quiver representations.
Our techniques are both explicit and functorial.
\end{abstract}

\section*{Introduction}
Vaughan Jones has revealed fascinating connections between Richard Thompson's groups $F\subset T\subset V$ and subfactor theory while trying to reconstruct conformal field theories \cite{Jones17,Brothier20}.
This led to Jones' technology\,---\,a powerful and explicit method for constructing group actions \cite{Jones18}, see also \cite[Section 2]{Brothier23}.
It was then observed that a certain class of unitary representations of $V$ are in one-to-one correspondence with the representations of the so-called Pythagorean $C^*$-algebra $\cP_2$ \cite{Brothier-Jones19}.
Additionally, all these representations extend to the Cuntz(--Dixmier) $C^*$-algebra $\cO_2$ (for the usual Birget--Nekrashevych embedding $V\into \cO_2$), and moreover all representations of $\cO_2$ arise in this way \cite{Dixmier64,Cuntz77,Birget04,Nekrashevych04}.
From this, the authors have forged powerful tools for studying representations of $\cO_2$ via representations of $\cP_2$ obtaining, as a complete surprise, moduli spaces of representations of $\cO_2$ \cite{Brothier-Wijesena22,Brothier-Wijesena24a,Brothier-Wijesena24b}.

This article studies representations of the celebrated Leavitt algebras \cite{Leavitt56} and their generalisation as the Leavitt \emph{path} algebras \cite{Abrams-Pino05,Ara-Moreno-Pardo07}
for which, until this article, the representation theory has remained very obscure and in ``an infant stage'' \cite{Rangaswamy20}.
We adopt a similar strategy as above where $\cP_2$ and $\cO_2$ are now replaced by the quiver algebra and the Leavitt path algebra, respectively, and complex representations of $C^*$-algebras are replaced by representations of algebras over arbitrary fields. 
Techniques have been appreciably changed to adapt to this algebraic and non-semisimple situation.

We now explain our approach and main results. We start by doing this for our ``base case" - the Leavitt algebras. Then, we explain how it readily generalises to Leavitt \emph{path} algebras over arbitrary \emph{row-finite} graphs $\Gamma$. 
We finish by confronting our results to previous ones in the literature.
The core of this article follows a similar presentation.

Fix $n\geq 2$, a ground field $\bk$, and consider the free algebra $A$ over $n$ generators and $L$ the Leavitt algebras with $2n$ generators. 
They are, respectively, the quiver algebra and the Leavitt path algebra of a bouquet of $n$ loops.
Given a representation or module $V\in \Rep(A)$ of $A$ (we will use both terms interchangeably) we form, via a direct limit, a new vector space $\Pi(V)$ with an $L$-module structure on it. 
This procedure is inspired by Jones' technology (see Remark \ref{rem:Jones-technology}) and is both explicit and functorial yielding $\Pi:\Rep(A)\to\Rep(L).$
There is an obvious algebra morphism $\iota:A\to L$ yielding a functor $\Lambda:\Rep(L)\to \Rep(A)$.
We prove $\Pi\circ \Lambda$ is naturally isomorphic to the identity functor. 
In particular, \emph{all} $L$-module can be constructed by an $A$-module using $\Pi$.

We then investigate when the restriction of $\Pi$ is fully faithful.
First, we define \emph{complete} submodules $V_c\subset V$ - those inducing an isomorphism $\Pi(V_c)\simeq \Pi(V)$.
Second, we say that $V$ is in the class $\cS$ when it admits a \emph{smallest} complete submodule $V_s\subset V$. 
Class $\cS$ is rather large and contains all finite-dimensional $A$-modules.
The process $\Sigma:V\mapsto V_s$ is functorial on $\cS$.
An $A$-module $V$ is \emph{full} when $V_s=V$.
Third, there is a canonical $A$-module map $j_V:V\to \Lambda\circ \Pi(V)$, and we say that $V$ is \emph{nondegenerate} when $j_V$ is injective. We observe that $\nabla:V\mapsto V/\ker(j_V)$ is functorial, produces nondegenerate modules, and commutes with $\Sigma$.
Our main theorem shows that $\Pi$ 
is fully faithful when restricted to nondegenerate full $A$-modules having $\Sigma\circ\Lambda$ as a weak inverse.

\begin{letterthm}\label{letterthm:main}(Theorem \ref{theo:main})
Let $\Rep_{\full}^{\nd}(A)$ be the category of full and nondegenerate representations of $A$ and let $\Rep_{\cS}(L)$ be the full subcategory of $\Rep(L)$ with class of objects $\{\Pi(V):\ V\in\cS\}$.
The functors
$$\Pi_{\full}^{\nd}:\Rep_{\full}^{\nd}(A)\to \Rep_{\cS}(L) \text{ and } \Sigma\circ\Lambda_\cS:\Rep_{\cS}(L)\to \Rep_{\full}^{\nd}(A)$$
define an equivalence of categories. 
\end{letterthm}

The key of the proof is to show that the $L$-module $\Pi(V)$ does remember the $A$-submodule $j_V(V_s)$ when $V\in \cS$. This was a great surprise and the main conceptual result of our previous work in the context of $C^*$-algebras \cite{Brothier-Wijesena24b}.
In this algebraic context the proof is much more elementary and reduces to showing that $\cS$ is closed under $\Lambda\circ\Pi$.

In this non-semisimple context we cannot deduce from Theorem \ref{letterthm:main} that $\Pi$ preserves irreducibility nor indecomposability. 
Nevertheless, the following result does permit to efficiently construct and classify irreducible and indecomposable representations of $L$.

\begin{letterthm}\label{letterthm:irred}(Theorem \ref{theo:irred})
If $V$ is an irreducible (resp.~indecomposable nondegenerate and full) representation of $A$, then $\Pi(V)$ is an irreducible (resp.~indecomposable) representation of $L$.
Moreover, if $V,V'\in \Rep(A)$ are irreducible and nonzero (resp.~indecomposable nondegenerate and full), then $V\simeq V'$ if and only if $\Pi(V)\simeq \Pi(V')$.
\end{letterthm}

We then introduce a novel and easily computable invariant:
the \emph{$A$-dimension} $\dim_A$ which is defined for $A$-modules and $L$-modules. 
The $A$-dimension is invariant under all the functors we have defined: $\Pi,\Lambda,\nabla, \Sigma$. 
Hence, given $W\in\Rep_{\cS}(L)$ and taking any $V\in\Rep(A)$ satisfying $\Pi(V)\simeq W$, we obtain that $\dim_A(W)$ is equal to the usual dimension of $j_V(V_s)$. 
Observe, \emph{all} representations of $L$ have infinite \emph{usual} dimension while many of them have finite $A$-dimension, see below.

Consider $\Irr(L)_d$ - the irreducible classes of representations of $L$ of $A$-dimension $d$. By our main theorem, $\Irr(L)_1$ is in bijection with the punctured vector space $\bk^{n}\setminus\{0\}$, while for all finite $d\neq 1$ we get that $\Irr(L)_d$ is in bijection with the irreducible classes of representations of $A$ of \emph{usual} dimension $d$.
Now, there are famous geometric structures for classes of representations of $A$: Nakajima's quiver varieties when $\bk=\C$ and King's moduli spaces for algebraically closed $\bk$ of characteristic $0$ \cite{Nakajima94,King94}. 
From this we deduce moduli spaces of representations of $L$.

\begin{lettercor}\label{lettercor}(Theorem \ref{theo:moduli-space})
Assume that $\bk$ is algebraically closed with characteristic $0$.
For finite $d\geq 1$ the set $\Irr(L)_d$ of irreducible classes of representations of $L$ of $A$-dimension $d$ is in bijection with a smooth variety of dimension $(n-1)d^2+1$.
\end{lettercor}

Observe that our base case is very far from satisfying the assumptions of Gabriel or Kac's theorems\,---\,$A$ has many representations (and so does $L$) \cite{Gabriel72,Kac80}. 
The irreducible classes of Chen modules and their generalisations of $A$-dimension $d$ forms a tiny piece of $\Irr(L)_d$: the (generalised) Chen modules form a one-dimensional sub-variety when $\bk$ is algebraically closed with characteristic $0$; and for arbitrary $\bk$ they are parameterised by at most $d$ coefficients out of $(n-1)d^2+1$ for $\Irr(L)_d$, see Section \ref{sec:Chen-module-few}. 

We now explain how to generalise these results to Leavitt \emph{path} algebras.
Consider \emph{any row-finite} directed graph $\Gamma$ and construct its associated quiver algebra $A$ and Leavitt path algebra $L$. 
In particular, $\Ga$ can be taken to be any locally finite graph.
The whole strategy and main results continue to hold: we can define complete, full, nondegenerate $A$-modules, the class $\cS$, the functors $\Pi,\Lambda,\Sigma,\nabla$, the $A$-dimension for representations of $A$ and $L$, and deduce an equivalence of categories via a restriction of $\Pi$ that preserves irreducibility and indecomposability.
Only now, most objects are decomposed over the vertex set $E^0$ of $\Ga$, including the $A$-dimension function.
Once our machinery is put into place we may use geometric invariant theory applied to quiver representations to deduce moduli spaces of representations of $L$ \cite{Mumford-Fogarty-Kirwan94,Nakajima94,King94}.

\begin{letterthm}\label{letterthm:main-graph}
(Section \ref{sec:path-algebra})
Let $\Ga$ be a \emph{row-finite} directed graph and let $A$ and $L$ be the associated quiver algebra and Leavitt path algebra over a field $\bk$.
\begin{enumerate}
\item The functors $\Pi^{\nd}_{\full}:\Rep^{\nd}_{\full}(A)\to \Rep_\cS(L)$ and $\Sigma\circ\Lambda_\cS:\Rep_\cS(L)\to\Rep_{\full}^{\nd}(A)$ define an equivalence of categories.
\item If $V\in \Rep(A)$ is irreducible (resp. indecomposable nondegenerate and full), then $\Pi(V)\in \Rep(L)$ is irreducible (resp. indecomposable).
Moreover, if $V,V' \in \Rep(A)$ are irreducible and nondegenerate (resp. indecomposable nondegenerate and full), then $V \simeq V'$ if and only if $\Pi(V) \simeq \Pi(V')$.
\item Assume that $\bk$ is algebraically closed of characteristic $0$.
The space $\Irr(L)_d$ of irreducible classes of representations of $L$ of fixed $A$-dimension $d=(d_\nu)_{\nu\in E^0}$ with $\sum_\nu d_\nu$ finite is either empty or in bijection with a smooth variety of dimension
\[\sum_{e \in E^1} d_{se}d_{re} - \sum_{\nu \in E^0} d_\nu^2 +1.\]
\end{enumerate}
\end{letterthm}

We end the paper by comparing our work with previous results in the literature.
We recall Chen modules and their generalisations which before this work were the main source of irreducible $L$-modules \cite{Chen15,Ara-Rangaswamy14,Rangaswamy15,Hazrat-Rangaswamy16, Anh-Nam21, Vas23}.
We explain how to construct these $L$-modules using $A$-modules (in most cases being finite dimensional) and extend some of these constructions yielding new explicit families of irreducible $L$-modules.
Then, we mention a series of work of Ko\c{c} and \"Ozaydin where they construct similar functors between $\Rep(A)$ and $\Rep(L)$ to ours \cite{Koc-Ozaydin18,Koc-Ozaydin20,Koc-Ozaydin22}. 
We conclude by explaining how their results and applications are rather orthogonal to ours.

This present paper has been written in a locally trivial fashion: we introduce a few key notions (completeness, fullness, nondegeneracy) and constructions (the functors $\Pi,\Lambda,\Sigma,\nabla$) followed by elementary observations. Nevertheless, this leads to powerful and surprising results - the miracle being that many $L$-modules encode a canonical $A$-module. 

\section{Definitions of the main objects}

Fix a field $\bk$ and assume all vector spaces and algebras are over $\bk$.
Also fix a natural number $n\geq 2$ and will consider $n$-ary trees and forests, free algebras over $n$ generators, and the $n$th Leavitt algebra.
A representation $(\pi,V)$ of an algebra $B$ (also called $B$-module) is an algebra morphism $\pi:B\to \End(V)$ into the endomorphisms of $V$.
Following common usages, we consider \emph{right}-modules, writing $v\cdot b$ for the image of $v\in V$ under $\pi(b)$ for $b\in B$.
All module are assumed to be unital in the following weak sense: the $v\cdot b$ with $v\in V,b\in B$ span $V$.
Representations or modules are called irreducible or simple if they do not contain any nonzero proper subrepresentations.
We write $\delta_{j,k}$ for the Kronecker symbol.

\subsection{Trees and forests}
A \emph{tree} $t$ is a finite rooted complete $n$-ary tree.
Thus, vertices have either $n$ or zero children; 
the latter called \emph{leaves} forming the set $\Leaf(t)$.
We view a tree as a graph embedded in the plane with the root on top and leaves on the bottom. 
Children of a vertex are ordered (left to right) and the leaves inherit an order. 

A \emph{forest} $f$ is a finite collection of trees $f=(f_1,\dots,f_r)$ where $f_j$ is the $j$th tree of $f$ and roots are ordered from left to right.
If $f$ is a forest with $r$ leaves and $g$ a forest with $r$ roots, then form $f\circ g$ by stacking $f$ on top of $g$ lining the $k$th leaf of $f$ with the $k$th root of $g$.
Moreover, given any pair of forests $(f,h)$ form $f\otimes h$ by concatenating the list of trees of $f$ and $h$ (or concatenating the corresponding graphs in the plane placing $f$ to the left of $h$). For instance, $f=(f_1,\dots,f_r)=f_1\ot\dots\ot f_r$.

We write $\cF$ for the collection of all forests and $\cT$ for the collection of all trees.
This gives a monoidal category $(\cF,\circ,\ot)$ with collection of objects $\N$.
Hence, an arrow or morphism $f:l\to r$ of $\cF$ corresponds to a forests with $l$ leaves and $r$ roots.
Moreover, $\cT$ is a directed set for the order $\leq$ defined as $t\leq s$ if and only if there exists a forest $f$ satisfying $s=t\circ f.$ In that case we say $s$ is obtained from \textit{growing} the tree $t$.
We write $\Y$ for the unique tree with $n$ leaves which we also call a \emph{caret} or the $n$-ary caret, and $I$ for the \emph{trivial tree} having a single vertex and no edges (this vertex being simultaneously a root and a leaf).
Denote $f_{j,r}$ for the forest, that we call \emph{elementary}, with $r$ roots and a single caret attached to its $j$th root with $1\leq j\leq r$. 
Every forest is a finite composition of elementary ones. This decomposition is not unique; we have the relations:
\begin{equation}\label{eq:Thompson}f_{j,r}\circ f_{k,r+n-1} = f_{k,r} \circ f_{j+n-1, r+n-1} \text{ for } 1\leq k<j\leq r.\end{equation}
This yields a presentation of the category $\cF$: the generators being the $f_{j,r}$ subject to the relation \eqref{eq:Thompson}.

\subsection{Free algebra and Leavitt algebra}
Define $A:=A_\bk(n)$ as the free algebra in the variables $e_1,\dots,e_n$, that is, the algebra of noncommutative polynomials over the indeterminates $e_1,\dots,e_n$, or the path algebra over the graph with one vertex and $n$ loops (a bouquet of $n$ loops).
Define $L:=L_\bk(n)$ to be the \emph{Leavitt algebra} with generator set
$$e_1,\dots,e_n, e_1^*,\dots,e_n^*$$
satisfying the relations
$$\sum_{i=1}^n e_ie_i^*=1 \text{ and } e_j^*e_k=\delta_{j,k} \text{ for all } 1\leq j,k\leq n.$$
Here the $e_i^*$ are indeterminates and are not adjoint of the $e_i$ (adjoints do not have obvious meaning in this algebraic context).
Define the following algebra morphism 
$$\iota:A\to L, e_i\mapsto e_i$$ 
where we abuse notation.
If $p=p_1\dots p_k$ is a finite product of $e_i$'s, then we write $p^*$ for $p_k^*\dots p_1^*$ - the order has been reversed.
We will often identify such $p$ with vertices in the infinite rooted complete $n$-ary tree or with leaves of a tree.

\begin{remark}
	When $\bk=\C$, then there is an algebra morphism with dense range from $L$ to the Dixmier--Cuntz algebra $\cO_n$ obtained by sending the $e_i,e_j^*$ to the usual generators of $\cO_n$ and their adjoints \cite{Dixmier64,Cuntz77}. A similar morphism can be construct from $A$ to the $n$th Pythagorean algebra $P_n$ \cite{Brothier-Jones19}. 
	These morphisms together with $\iota:A\to L$ and the quotient map $P_n\onto\cO_n$ form a commutative diagram.
\end{remark}

We write $\Rep(A)$ and $\Rep(L)$ for the usual categories of \emph{unital} representations of $A$ and $L$. 
The algebra morphism $\iota:A\to L$ defines a forgetful functor 
\begin{equation}\label{eq:forget}\Lambda:\Rep(L)\to\Rep(A),\ (\sigma,W)\mapsto (\sigma\circ\iota,W).\end{equation}

\subsection{Isomorphism of categories with $\Rep(A)$ and certain functors.}

Consider the direct sum $\oplus$ of vector spaces as a monoidal product of the category $\Vect$ of vector spaces.
Let $\Func_{\ot}(\cF,\Vect)$ be the category of (contravariant) monoidal functors from $(\cF,\otimes)$ to $(\Vect,\oplus)$.
We have an isomorphism of categories
$$\Caret:\Rep(A)\to \Func_{\ot}(\cF,\Vect)$$
defined as follows.
Given $V\in \Rep(A)$ we set $\Caret(V)$ to be the unique monoidal contravariant functor from $\cF$ to $\Vect$ sending the object $1$ to $V$ and sending the $n$-ary caret $\Y$ to the mapping $V\ni v\mapsto (v\cdot e_1,\dots,v\cdot e_n) \in V^{\oplus n}.$
By monoidality, this functor sends the object $k$ to $V^{\oplus k}$ and the elementary forest $f_{j,n}$ to the morphism $\id^{\oplus j-1}\oplus \Caret(V)(\Y)\oplus \id^{\oplus n-j}.$
The universality of the category $\cF$ (with respect to elementary forests) proves that there exists indeed one and only one such functor.
An $A$-equivariant morphism $\theta:V\to \ti V$ defines a natural transformation in the obvious way via the $\theta^{\oplus k}:V^{\oplus k}\to (\ti V)^{\oplus k}.$ 

Conversely, if $\Phi:\cF\to\Vect$ is a monoidal functor, then it is completely determined by $\Phi(1)$ ($1$ denoting the corresponding object of $\cF$) and by $\Phi(\Y)$.
Taking coordinates of the mapping $\Phi(\Y):\Phi(1)\to\Phi(1)^{\oplus n}$ yields an $A$-module structure on $\Phi(1).$
This defines a functor in the other direction $\Func_{\ot}(\cF,\Vect)\to \Rep(A)$ which is the inverse of $\Caret$.

\subsection{From $\Rep(A)$ to $\Rep(L)$.}

We now explain one of the most important construction in this paper\,---\,the functor
$$\Pi:\Rep(A)\to \Rep(L).$$
Consider $(\pi,V)\in\Rep(A)$ and write $\Phi$ for the functor $\Caret(V).$
Define 
$$W:=\Pi(V)=\{(t,v):\ t\in\cT, v:\Leaf(t)\to V\}/\sim$$
where $\cT$ denotes the set of all finite rooted $n$-ary trees and where the equivalence relation $\sim$ is generated  by 
$$(t,v)\sim (t\circ f,\Phi(f)(v)).$$
Write $[t,v]$ for the class of $(t,v)$ under $\sim.$
The space $W$ is then composed of classes of decorated trees $[t,v]$ where $v$ is a tuple of vectors of $V$ indexed by the leaves of $t$.
When we grow trees we then apply the caret maps $\Phi(\Y)$ a finite number of times on the decorations.
Equivalently, $W$ is the direct limit of the directed systems of vectors spaces $V_t:=V^{\Leaf(t)}$ with connecting maps 
$$\iota_{t}^{t\circ f}: V_t\to V_{t\circ f}, (t,v)\mapsto (t\circ f, \Phi(f)(v)).$$

Now, we define an action of $L$ on $W$.
Consider $[t,v]\in W$ and decompose the tree $t$ into $t=\Y\circ (t_1,\dots,t_n)$. 
We put $$ [t,v]\cdot e_j:=[t_j,v_j]$$ 
where $v_j:\Leaf(t_j)\to V$ is the vector associated to the subtree $t_j$ of $t$. 
This consists in snipping the tree at the $j$th leaf.
Moreover, we set $$[s,w] \cdot e_j^*= [\Y\circ (I^{\ot j-1} \ot s \ot I^{\ot n-j} , 0^{\oplus j-1} \oplus w \oplus 0^{\oplus n-j}].$$
This consists in gluing the decorated tree $(s,w)$ at the $j$th leaf of the caret $\Y$ and by decorating the remaining leaves of $\Y$ with the zero vector.

A routine computation permits to check that we have defined an $L$-module structure.
Now, given a morphism of $A$-module $\theta:V\to V'$ we consider the formula
$$\Pi(\theta):[t,v_1,\dots,v_k]\mapsto [t,\theta(v_1),\dots,\theta(v_k)]$$
where here $t$ is a tree with $k$ leaves and $v_i\in V$ is the decoration of its $i$th leaf.
By construction we easily verify $\Pi$ defines a functor which moreover preserves direct sums.

\begin{proposition}
Consider $\Rep(A)$ and $\Rep(L)$ equipped with the direct sum for monoidal product.
The process $\Pi$ described above defines a monoidal functor
$$\Pi:\Rep(A)\to \Rep(L).$$
\end{proposition}

\begin{example}   
Below illustrates the action of $e_1e_2$ on $[\wedge, v_1, v_2]$ for $n=2$. The tree is first ``grown'' by adding a caret, and then ``snipped'' at the vertex given by the path $e_1e_2$.

\begin{center}
	\begin{tikzpicture}[baseline=0cm, scale = 1]
		\draw (0,0)--(-.5, -.5);
		\draw (0,0)--(.5, -.5);
		
		\node[label={[yshift=-22pt] \normalsize $v_1$}] at (-.5, -.5) {};
		\node[label={[yshift=-22pt] \normalsize $v_2$}] at (.5, -.5) {};
		
		\node[label={[yshift= -3pt] \normalsize }] at (1.7, -.7) {$\sim$};
	\end{tikzpicture}%
	\begin{tikzpicture}[baseline=0cm, scale = 1]
		\draw[thick] (0,0)--(-.5, -.5);
		\draw (0,0)--(.5, -.5);
		\draw (-.5, -.5)--(-.9, -1);
		\draw[thick] (-.5, -.5)--(-.1, -1);
		
		\node[label={[yshift=-22pt] \normalsize $v_1 \cdot e_1~~~~$}] at (-.9, -1) {};
		\node[label={[yshift=-22pt] \normalsize $~~~~v_1 \cdot e_2$}] at (-.1, -1) {};
		\node[label={[yshift=-22pt] \normalsize $v_2$}] at (.5, -.5) {};
		
		\node[label={[yshift= -3pt] \normalsize $e_1e_2$}] at (1.4, -.7) {$\longmapsto$};
	\end{tikzpicture}%
	\begin{tikzpicture}[baseline=0cm, scale = 1]
		\node[label={[yshift=-25pt] \normalsize $v_1 \cdot e_2$}] at (0, -.5) {$\bullet$};
	\end{tikzpicture}.%
\end{center}

Below illustrate the action of $e_1^*$ on $[\Y,v_1,v_2]$:

\begin{center}
	\begin{tikzpicture}[baseline=0cm, scale = 1]
		\draw[thick] (0,-.3)--(-.5, -.8);
		\draw[thick] (0,-.3)--(.5, -.8);
		
		\node[label={[yshift=-22pt] \normalsize $v_1$}] at (-.5, -.8) {};
		\node[label={[yshift=-22pt] \normalsize $v_2$}] at (.5, -.8) {};
		
		\node[label={[yshift= -3pt] \normalsize $e^*_{1}$}] at (1.4, -.7) {$\longmapsto$};
	\end{tikzpicture}%
	\begin{tikzpicture}[baseline=0cm, scale = 1]
		\draw (0,0)--(-.5, -.5);
		\draw (0,0)--(.5, -.5);
		\draw[thick] (-.5, -.5)--(-.9, -1);
		\draw[thick] (-.5, -.5)--(-.1, -1);
		
		\node[label={[yshift=-22pt] \normalsize $v_1$}] at (-.9, -1) {};
		\node[label={[yshift=-22pt] \normalsize $v_2$}] at (-.1, -1) {};
		\node[label={[yshift=-22pt] \normalsize $0$}] at (.5, -.5) {};		
	\end{tikzpicture}%
\end{center}
\end{example}

\begin{definition}
Let $j_V:V\to W$ be the \emph{canonical map} $v\mapsto [I,v]$ (writing $j$ for $j_V$ in clear context).
\end{definition}
By construction we observe that $j$ defines a morphism of $A$-module from $V$ to $\Lambda(\Pi(V))$ (see \eqref{eq:forget} for the definition of $\Lambda$).
Moreover, the collection $(j_V:\ V\in \Rep(A))$ defines a natural transformation from $\id$ to $\Lambda\circ \Pi$.

\begin{remark}\label{rem:Jones-technology}
Originally, given a functor $\Phi:\cF\to \cC$ in a category $\cC$ when $n=2$ Jones defined an action of the fraction groupoid of $\cF$ that he restricted to its first isotropy group. This yielded an action of Thompson's group $F$ \cite{Jones18}.
It was then observed that if $\Phi$ was monoidal with $\cC$ the category of Hilbert spaces with linear isometries for morphisms and direct sum for monoidal product, then the this action of $F$ was a unitary representation that uniquely extended to a representation of the Cuntz algebra \cite{Brothier-Jones19}.

Similarly, we observe that a monoidal functor $\Phi:\cF\to\Vect$ defines a representation of $F$ which uniquely extends to a representation of $L$. 
Then, the isomorphism of categories $\Caret:\Rep(A)\to\Func_\otimes(\cF,\Vect)$ yields $\Pi:\Rep(A)\to\Rep(L)$.
\end{remark}

\begin{proposition} \label{prop:jones-functor-surj}
If $W\in \Rep(L)$, then the canonical map $j_W:W\to \Pi(W)$ defines an isomorphism of $L$-modules.
Moreover, the collection $(j_W:W\in\Rep(L))$ defines a natural isomorphism from $\id$ to $\Pi\circ \Lambda.$
In particular, any $W\in \Rep(L)$ is isomorphic to $\Pi(V)$ for a certain $V\in \Rep(A).$
\end{proposition}

\begin{proof}
Fix an $L$-module $W$ that we consider as a $A$-module via $\iota:A\to L$.
Write $\widehat W$ for the $L$-module $\Pi(W).$
Let us show that the canonical map $j:W\to \widehat W, w\mapsto [I,w]$ defines an isomorphism of $L$-modules. 
Note that $w\mapsto (w\cdot e_1,\dots,w\cdot e_n)$ is a bijection from $W$ to $W^n$ with inverse $(w_1,\dots,w_n)\mapsto \sum_i w_i\cdot e_i^*.$
This implies that for any tree $t$ we have that 
$$\Phi(t)w = (w\cdot e_\ell)_{\ell\in \Leaf(t)}$$ 
is bijective from $W$ to $W^{\Leaf(t)}$ with inverse map $$(w_\ell)_\ell\mapsto \sum w_\ell\cdot e_\ell^*.$$
We deduce that the formula
$$[t,w]\mapsto \sum_{\ell\in \Leaf(t)} w_\ell\cdot e_\ell^*$$
defines a linear bijection from $\widehat W$ to $W$.
Moreover, we observe that it is an inverse of $j$ implying that $j$ is an isomorphism of $A$-modules.
Moreover, observe that 
\begin{align*}
j(w\cdot e_i^*) & = [I,w\cdot e_i^*] = [\Y, \Phi(\Y)(w\cdot e_i^*)]\\
& = [\Y, w\cdot e_i^* e_1, \dots, w\cdot e_i^* e_n] \\
& = [\Y , 0 ,\dots, 0 , w, 0, \dots, 0] \text{ with $w$ at the $i$th spot}\\
& = [I,w]\cdot e_i^*
\end{align*}
proving $L$-modularity of $j$. Thus, we have shown $j$ realises an isomorphism of $L$-modules between $W$ and $\Pi(W)$. 
Naturality follows from an elementary computation.
\end{proof}

\section{Complete, full and degenerate representations} \label{sec:base-case}
We now introduce our main key concepts. 
Completeness and fullness are algebraic analogues of previous notions defined by the authors for modules over $C^*$-algebras \cite{Brothier-Wijesena24b}.
Degeneracy is a new concept proper to this algebraic situation.
In this section we fix $(\pi,V)\in \Rep(A)$ and set $\Phi:=\Caret(V)$ and $W:=\Pi(V).$
Recall that the $e_i,1\leq i\leq n$ are the generators of $A$. If $p=p_1\dots p_k$ is a word in the $e_i$', then it has \emph{length} $|p|:=k.$
Such $p$ are often called \emph{words, products}, or \emph{paths}.

\subsection{Complete subrepresentations}
\begin{definition} \label{def:complete-subrep}

\begin{enumerate}
\item A subrepresentation $V_c\subset V$ is called \emph{complete} if for all $v\in V$, there exists $k\geq 1$ satisfying $v\cdot p\in V_c$ for all paths $p$ with $|p|\geq k$.
Equivalently, $V_c\subset V$ is complete if each $v\in V$ admits a tree $t\in\cT$ satisfying $\Phi(t)(v)\in V_c^{\Leaf(t)}$.
\item Write $\Comp(V)$ for the poset of all complete subrepresentations of $V$ ordered by inclusion.
\item Write $\cS$ for the class of representations $V'$ of $A$ that admits a \emph{smallest} complete subrepresentation denoted $V_s'$. 
Write $\Rep_\cS(A)$ for the associated full sub-category. 
\end{enumerate}
\end{definition}

We start with a reformulation of completeness which motivates our choice of terminology.
\begin{observation}\label{obs:complete}
A subrepresentation $V_c\subset V$ is complete if and only if the subspace $j(V_c)\subset W$ generates the $L$-module $W$, i.e.~the smallest $L$-submodule of $W$ containing $j(V_c)$ is $W$. 
This means that the inclusion map $\epsilon:V_c\into V$ defines an isomorphism $\Pi(\epsilon):\Pi(V_c)\to \Pi(V)$ of $L$-modules.
\end{observation}

We describe completeness using block matrix decompositions.
\begin{remark} \label{rem:submodule-decomp}
If $V_c\subset V$ is a submodule, then we can find a basis of the vector space $V$ so that each $\pi(e_i)$ decomposes into the following block matrix over the direct sum decomposition $V=V_c\oplus X_c$ ($X_c$ being an algebraic complement of $V_c$):
$$\begin{pmatrix}
Z_{i} & B_i \\ 
0 & N_i
\end{pmatrix}.$$
The submodule is complete when there exists a natural number $k\geq 1$ satisfying that any product $N_{i_1}\dots N_{i_k}$ with $k$ factors equals $0$.
In particular, all the $N_i$ are nilpotent.
Conversely, a family of such matrices defines a complete submodule.
\end{remark}

\begin{proposition} \label{prop:complete-rep}
Consider $\theta:V\to V'$ a morphism of $A$-modules.

\begin{enumerate}
\item If $V_c'\subset V'$ is complete, then so is $\theta^{-1}(V_c')\subset V$.
\item If $V_c\subset V$ is complete, then so is $\theta(V_c)\subset \theta(V)$.
\item If $k\geq 2$, $V_1\subset\dots\subset V_k$ is a chain of $A$-modules so that $V_i\subset V_{i+1}$ is complete for $1\leq i\leq k-1$, then $V_1\subset V_k$ is complete.
\item If $V_c\subset V$ is complete, then so is $j(V_c)\subset \Lambda\circ \Pi(V)$.
\item The poset $\Comp(V)$ is closed under taking finite intersections.
\end{enumerate}
\end{proposition}
\begin{proof}
Proof of (1).
Fix $v\in V$. Since $V_c'\subset V'$ is complete there exists $k\geq 1$ so that for any product $p$ with $|p|\geq k$ we have $\theta(v)\cdot p\in V_c'.$
Hence, $\theta(v\cdot p)\in V_c'$ and thus $v\cdot p\in \theta^{-1}(V_c').$

Proof of (2).
Consider $\theta(v)\in \theta(V)$.
By completenss of $V_c\subset V$ there exists $k\geq 1$ so that if $|p|\geq k,$ then $v\cdot p\in V_c.$
Therefore, $\theta(v)\cdot p\in \theta(V_c)$ proving completeness of $\theta(V_c)\subset\theta(V).$

Proof of (3).
For any $2\leq i\leq k$ there exists $n_i\geq 1$ satisfying $ v_i\cdot p\in V_{i-1}$ for all $v_i\in V_i$ and word $|p|\geq n_i$.
Now, if $v\in V_k$ and a word $p$ has length larger than $\sum_{i=2}^k n_i$, then $v\cdot p\in V_1$ proving that $V_1\subset V_k$ is complete.

Proof of (4).
Consider $[t,v]\in \Pi(V)$ where $t\in\cT$ and $v=(v_\ell)_{\ell}$ is a tuple of elements of $V$ indexed by the leaves $\ell$ of $t$.
Observe that $ [t,v]\cdot e_\ell=[I,v_\ell]\in j(V)$ for all $\ell\in \Leaf(t)$. 
Hence, $j(V)\subset \Lambda\circ\Pi(V)$ is complete by Observation \ref{obs:complete}.
Now, $V_c\subset V$ is complete which together with (2) imply that $j(V_c)\subset j(V)$ is complete.
We conclude using (3) and the chain $j(V_c)\subset j(V)\subset \Lambda\circ\Pi(V)$.

Proof of (5).
Consider some complete submodules $V_c^1,\dots,V_c^m$ in $V$ and their intersection $V_c.$
Take $v\in V$. For each $1\leq i\leq m$ there exists $n_i\geq 1$ so that $v\cdot p\in V_c^i$ for all word $p$ with $|p|\geq n_i$.
If $n:=\max(n_j:\ 1\leq j\leq m)$ and $|p|\geq n$, then $v\cdot p\in V_c$ proving completeness.
\end{proof}

We now derive properties concerning the class $\cS$.

\begin{proposition}\label{prop:class-S}
Consider $V$ in $\cS$ with smallest complete subrepresentation $V_s\subset V$.
\begin{enumerate}
\item The subrepresentation $V_s$ is equal to the intersection of all the complete ones.
\item If $V\subset V'$ is a complete subrepresentation, then $V'\in\cS$ and $V_s\subset V'$ is the smallest complete subrepresentation.
\item If $V,V'\in \cS$ and $\theta:V\to V'$ is a morphism of $A$-modules, then $\theta(V_s)\subset V_s'$ where $V_s'\subset V'$ is the smallest complete subrepresentation.
Moreover, $\theta(V_s)\subset \theta(V)$ is the smallest complete submodule.
\item We have that $\Lambda\circ\Pi(V)$ is in $\cS$ with smallest complete submodule $j(V_s)$.
\item If $V$ is finite dimensional, then $V\in \cS$.
\end{enumerate}
\end{proposition}

\begin{proof}
Proof of (1). This follows from $\Comp(V)$ being closed under finite intersections.

Proof of (2).
The smallest complete submodule $V'_s$ of $V'$ is the intersection of all complete submodules of $V'$ by (1).
Since $V\subset V'$ is complete we may restrict to the intersection of all complete submodules of $V'$ contained in $V$. This is precisely equal to $V_s$ by (1).

Proof of (3).
By (1) of Proposition \ref{prop:complete-rep} we have that $V_c:=\theta^{-1}(V'_s)$ is complete inside $V$ since $V_s'\subset V'$ is.
Hence, $V_s\subset V_c$ by minimality.
Therefore, $\theta(V_s)\subset \theta(V_c)\subset V'_s$. 
By (2) of Proposition \ref{prop:complete-rep}, $\theta(V_s)\subset \theta(V)$ is complete. 
If $Z_c\subset \theta(V)$ is complete, then so is $\theta^{-1}(Z_c)\subset V$ and thus $V_s\subset \theta^{-1}(Z_c)$ by minimality. This implies that $\theta(V_s)\subset Z_c$ and thus $\theta(V_s)\subset \theta(V)$ is the smallest complete submodule.

Proof of (4).
By (4) of Proposition \ref{prop:complete-rep} we have that $j(V)\subset \Lambda\circ\Pi(V)$ is complete.
By (3) $j(V_s)\subset j(V)$ is the smallest complete submodule. We can now conclude using (2).

Proof of (5).
This comes from the following elementary fact: given a finite dimensional vector space $V$ and $\cV$ a nonempty family of nonzero vector subspaces of $\cV$ that is closed under taking finite intersections, then $\cV$ admits a smallest element for the inclusion.
This can be proven by induction on $\dim(V)$. We then conclude using (5) of Proposition \ref{prop:complete-rep}.
\end{proof}

\subsection{Full representations}

\begin{definition} \label{def:full}
A representation $V\in \Rep(A)$ is \emph{full} if it does not admit any proper complete subrepresentation.
We write $\Rep_{\full}(A)$ for the associated full subcategory.
\end{definition}

Extracting the smallest complete submodule is a functorial process as proven below.

\begin{proposition}\label{prop:sigma}
There exists an essentially surjective functor
$$\Sigma:\Rep_\cS(A)\to \Rep_{\full}(A)$$
so that $\Sigma(V)=V_s$ is the smallest complete submodule of $V$ and 
$$\Sigma(\theta)=\theta\restriction_{V_s}\in \Hom(V_s,V_s')$$ is the restriction to $V_s$ for $\theta\in\Hom(V,V')$.
\end{proposition}

\begin{proof}
By definition, the process $V\mapsto V_s$ is well-defined from the class $\cS$ to the class of full modules.
By (3) of Proposition \ref{prop:class-S} a morphism $\theta:V\to V'$ restricts into a morphism $\theta_s:V_s\to V_s'$.
We deduce that $\Sigma$ is indeed a well-defined functor.
Observe that $\Sigma$ restricts to the identity endofunctor of $\Rep_{\full}(A)$ implying it is essentially surjective.
\end{proof}

We note that $\Sigma$ is not faithful (i.e.~not injective on all morphism spaces) and thus does not define an equivalence of categories.
Indeed, take $V$ one-dimensional with $v\cdot e_i=0$ for all $1\leq i\leq n$ and $v\in V$. We have that $\End(V)\simeq \bk$ while $\End(\Sigma(V))=\End(0)\simeq \{0\}$.

\subsection{Degenerate representations}\label{sec:degenerate}
Recall the notations $(\pi,V)\in \Rep(A), \Phi=\Caret(V)$ of this section.
It may happen that the canonical map $j:V\to \Pi(V)$ is \emph{not injective}.
Even worse, we may have $V\neq \{0\}$ while $\Pi(V)=\{0\}$, e.g.~take any vector space $V$ and all the $e_i$ acting as a single nilpotent linear map. 
This is new to this algebraic framework and does not appear in the $C^*$-algebraic context of \cite{Brothier-Wijesena24b}. This causes some technicalities that we treat via the following notion.

\begin{definition}
A representation $V$ of $A$ is called \emph{nondegenerate} when $j:V\to \Pi(V)$ is injective. Otherwise, it is called \emph{degenerate}.
\end{definition}

\begin{proposition}\label{prop:degenerate}
The following assertions are true.
\begin{enumerate}
\item The spaces $\ker(\Phi(\Y))$ is the set of $v\in V$ satisfying $v\cdot e=0$ for all $e\in E^1.$
\item The $A$-module $\ker(j)$ is the set of $v\in V$ admitting a tree $t$ satisfying $\Phi(t)(v)=0$. 
Equivalently, $v\in \ker(j)$ if and only if there exists $k\geq 1$ satisfying $v\cdot p=0$ for all word $p$ with $|p|\geq k.$
\item The representation $V$ is nondegenerate if and only if $\Phi(\Y)$ is injective.
\item If $W\in \Rep(L)$, then $\Lambda(W)$ is nondegenerate.
\end{enumerate}
\end{proposition}
\begin{proof}
Proof of (1). 
Since $\Phi(\Y)$ is the direct sum of the $v\mapsto v\cdot e$ over $e\in E^1$ we deduce the description of $\ker(\Phi(\Y)).$

Proof of (2).
The kernel of $j$ is a submodule of $V$ since $j:V\to \Lambda\circ\Pi(V)$ is a morphism of $A$-modules.
Consider $v\in V$.
If $j(v)=0$, then there exists a representative $(t,\Phi(t)(v))$ of $j(v)$ with all coordinates of $\Phi(t)(v)$ equal to zero. 
Up to growing $t$ we may assume that $t=t_k$ is the full tree with all leaves at distance $k$ from the root. 
We get that $\Phi(t_k)(v)=\oplus_p v\cdot p$ where the sum is over all path of length $k$.
The converse inclusions are obvious.

Proof of (3).
Assume now that $\ker(\Phi(\Y))=\{0\}$ and write $V_k$ for $\ker(\Phi(t_k)).$
If $v\in V_2$, then $\Phi(t_2)(v)=0$ for the full tree $t_2= \Y\circ (\Y\otimes\Y)$. 
Hence, the first and second coordinates of $\Phi(\Y)v$ are in $V_1 = \ker(\Phi(\Y))$ which is assumed to be trivial.
Therefore, $\Phi(\Y)v=0$, meaning $v\in V_1=\{0\}.$ 
By iterating this argument we deduce that $V_k=\{0\}$ for all $k\geq 1$. By (2) we have that $\ker(j)=\cup_{k\geq 1} V_k$ which is then trivial.

Proof of (4).
Observe that the caret map $\Psi(\Y)$ associated to $\Lambda(W)$ is a bijection from $W$ to $W^n$ with inverses $(w_1,\dots,w_n)\mapsto \sum_{i=1}^n w_i \cdot e_i^*$.
Hence, by (3) $\Lambda(W)$ is nondegenerate.
\end{proof}

\begin{remark}We have proven that if $\ker(\Phi(\Y))=\{0\}$, then $\ker(j)=\{0\}$.
This does not mean that $\ker(\Phi(\Y))=\ker(j)$ in general.
Indeed, take $n=2$, $V=\bk^2$ with basis $(E_1,E_2)$, and set $\pi(e_1)=\pi(e_2)=\begin{pmatrix} 0 & 1\\ 0& 0 \end{pmatrix}$.
We have that $\ker(\Phi(\Y))=\ker(\pi(e_1))= \bk\cdot E_1$ while $\ker(j)=V$.\end{remark}

We now define a (functorial) procedure transforming any $A$-module into a nondegenerate one while yielding an isomorphism between the associated $L$-modules.

\begin{proposition}\label{prop:degenerate-functor}
Consider $V\in \Rep(A)$ and the submodule $\ker(j_V)$.
The quotient map $V\onto V/\ker(j_V)=:\nabla(V)$ defines an isomorphism of $L$-modules $\Pi(V)\to\Pi(V/\ker(j_V))$.
Moreover, this defines an essentially surjective functor $\nabla:\Rep(A)\to\Rep^{\nd}(A)$ from $\Rep(A)$ to the full subcategory of \emph{nondegenerate} representations of $A$.
\end{proposition}

\begin{proof}
Consider $j_V:V\to\Pi(V)$ and write $K$ for $\ker(j_V)$ and $q:V\onto V/K$ for the quotient map. 
This is a morphism of $A$-modules and thus the functor $\Pi$ defines a morphism of $L$-modules
$$\Pi(q):\Pi(V)\to \Pi(V/K).$$
Let us show that $Q:=\Pi(q)$ is an isomorphism.
Observe that $Q$ is defined as follows:
$$Q([t,v_1,\dots,v_m])=[t,q(v_1),\dots,q(v_m)]$$
where $t$ is a tree and the $v_i\in V$ are the decorations of its leaves.
We may write $v$ instead of $(v_1,\dots,v_m)$ and write $q(v)$ instead of $(q(v_1),\dots,q(v_m)).$
Assume that $Q([t,v])=0$.
Hence, $[t,q(v)]=0$ meaning that there exists a forest $f$ satisfying that $(t\circ f, \Phi_{V/K}(f)(q(v)))=(t\circ f, 0)$ (where $\Phi_{V/K}=\Caret(V/K)$).
Up to growing $t$ we may assume that $f$ is trivial and thus obtain that 
$q(v_i)=0$ for all $1\leq i\leq m$.
This means that $v_i$ is in $\ker(j_V)$ for all $i$.
Hence, there exists a tree $s$ satisfying $\Phi(s)(v_i)=0$ for all $i.$
Take now the forest $g:=(s,\dots,s)$ with $m$ roots and whose each tree is equal to $s$.
We obtain that $$[t,v]=[t\circ g,\Phi(g)(v)]=[t\circ g, \Phi(s)(v_1),\dots,\Phi(s)(v_m)]=[t\circ g, 0,\dots,0]=0.$$
This proves that $Q$ is injective.
Since $q:V\to V/K$ is surjective we immediately deduce that $\Pi(q):\Pi(V)\to\Pi(V/K)$ is surjective.

Let us check that $V/K$ is nondegenerate.
Consider $v+K\in V/K$ and assume that it is in the kernel of $\Phi_{V/K}(\Y)$.
Observe that 
$$\Phi_{V/K}(\Y)(v+K) = \Phi(\Y)(v)+K.$$
We deduce that $\Phi(\Y)v$ is in $K$. Hence, there exists a large enough tree $s$ so that $\Phi(s)(w_i)=0$ for all the coordinates of $w:=\Phi(\Y)v$.
Hence, $t:=\Y\circ(s,\dots,s)$ is a tree satisfying that $\Phi(t)v=0$ and thus $v$ is in $K:=\ker(j_V)$ by Proposition \ref{prop:degenerate}.
Therefore, $v+K$ is trivial in $V/K$ and $\Phi_{V/K}(\Y)$ is injective.
Using Proposition \ref{prop:degenerate} we conclude that $V/K$ is nondegenerate.

It remains to check that $\nabla$ is indeed a functor.
It is sufficient to show that if $\theta\in\Hom(V,V')$, then it defines a morphism $[\theta]:\Hom(V/\ker(j_V), V'/\ker(j_{V'}))$ using the universal property of the quotient.
This reduces to showing that given $v\in \ker(j_V)$, we have $\theta(v)\in\ker(j_{V'})$.
Take such a $v$. We will use the characterisation of $\ker(j_V)$ of Proposition \ref{prop:degenerate}.
There exists $k\geq 1$ such that $v\cdot p=0$ for all words $|p|\geq k$.
Now, $0=\theta(0)=\theta( v\cdot p)= \theta(v)\cdot p$. Using again Proposition \ref{prop:degenerate} we deduce that $\theta(v)\in \ker(j_{V'})$. 
\end{proof}

We now show that irreducibility implies non-degeneracy and fullness except in one case.

\begin{proposition} \label{prop:degenerate-rep}
An \emph{irreducible} representation $(\pi,V)\in\Rep(A)$ is nondegenerate and full unless $\dim(V)=1$ and $\pi(e_i)=0$ for all $1\leq i\leq n$.
In this latter case $V$ is both degenerate and non-full. 
\end{proposition}

\begin{proof}
Consider an irreducible degenerate representation $(\pi,V)$ of $A$.
Consider the submodule $\ker(\Phi(\Y))$ where $\Phi:=\Caret(V)$.
Since $V$ is irreducible we have that $\ker(\Phi(\Y))=\{0\}$ or $V$.
Degeneracy and Proposition \ref{prop:degenerate} imply that $\ker(\Phi(\Y))\neq\{0\}$.
Hence, $\Phi(\Y)$ is the zero map, i.e.~$\pi(e_i)=0$ for all $i$.
This implies that any vector subspace of $V$ is closed under the $\pi(e_i)$ and thus forms a $A$-submodule.
Irreducibility forces to have $\dim(V)=0$ or $1$.
Degeneracy implies $V\neq\{0\}$ and thus $\dim(V)=1.$
Conversely, if $\dim(V)=1$ and $\pi(e_i)=0$ for all $i$, we have that $(\pi,V)$ is irreducible and degenerate.

Consider now any irreducible representation $V'$.
If $V_c'\subset V'$ is a complete submodule, then $V_c'=\{0\}$ or $V_c'=V'$ by irreducibility.
If $V_c'=\{0\}$, then we are in the previous case. If $V_c'=V'$ for all complete $V_c'$, then the representation is full.
\end{proof}

\subsection{Connections between fullness and nondegeneracy} \label{subsec:full-nondeg}

Since $\Rep(A)$ is not semisimple (i.e.~there exist representations that are not direct sums of irreducible ones) we cannot expect that  fullness implies nondegenerate or vice-versa as illustrated below.

\begin{example}
Here is a full module that is degenerate.
Take any field $\bk$, fix $n=2$ and $V=\bk^2$ with standard basis $E_1,E_2$. 
Consider the $A$-module structure given by 
$$\pi(e_1):=\begin{pmatrix} 0 & 1\\ 0&1\end{pmatrix} \text{ and } \pi(e_2):=\begin{pmatrix} 0 & 1\\ 0&0\end{pmatrix}.$$
We have that $\ker(\pi(e_1))\cap \ker(\pi(e_2))=\bk E_1$ implying that $V$ is degenerate.
However, there is no proper complete submodule implying that $V$ is full.

Here is a nondegenerate module that is not full.
Take again any field $\bk$, $n=2$ and now take $V=\bk^3$.
Define 
$$\pi(e_1):=\begin{pmatrix} 1 & 1 & 0 \\ 0 & 0 & 1 \\ 0 & 0 & 0 \end{pmatrix} \text{ and } \pi(e_2):=\begin{pmatrix} 1 & 0 & 0 \\ 0 & 0 & 1 \\ 0 & 0 & 0 \end{pmatrix}.$$
Note that $\bk E_1$ is a complete submodule and thus $V$ is not full.
However, $\ker(\pi(e_1))=\bk (E_1-E_2)$ while $\ker(\pi(e_2))=\bk E_1$. Therefore, $\ker(\pi(e_1))\cap \ker(\pi(e_2))=\{0\}$ implying that $V$ is nondegenerate by Proposition \ref{prop:degenerate}.
\end{example}

We have defined two functors $\nabla$ and $\Sigma$ that transform $A$-modules into nondegenerate and full ones, respectively.
These two processes functorialy commute. Hence, applying the two functors in any order provides a nondegenerate full module that is invariant under $\nabla$ and $\Sigma$.
More precisely, if $V\in \cS$, then $j_V(\Sigma(V))
\simeq \nabla\circ\Sigma(V)\simeq \Sigma\circ\nabla(V)$ and these isomorphisms are natural.

\section{Classification of representations of the Leavitt algebra}

We now prove our main results in the base case of the Leavitt algebra.

\subsection{An equivalence of categories}

\begin{theorem}\label{theo:main}
Consider $\Rep_{\full}^{\nd}(A)$ the category of full and nondegenerate representations of $A$ and let $\Rep_{\cS}(L)$ be the full subcategory of $\Rep(L)$ with class of objects $\{\Pi(V):\ V\in\cS\}$.
The functors
$$\Pi_{\full}^{\nd}:\Rep_{\full}^{\nd}(A)\to \Rep_{\cS}(L) \text{ and } \Sigma\circ\Lambda_\cS:\Rep_{\cS}(L)\to \Rep_{\full}^{\nd}(A)$$
define an equivalence of categories. 
\end{theorem}

\begin{proof}
We start by explaining why $\Sigma\circ \Lambda_\cS$ is well-defined.
Recall that $\Lambda_\cS$ is valued into \emph{nondegenerate} modules by Proposition \ref{prop:degenerate}.
Now, $W\in\Rep_\cS(L)$ means that $W\simeq\Pi(V)$ for some $V\in\cS$ and that $j_V(V_s)\subset \Lambda(W)$ is the smallest complete submodule (and thus exists) by Proposition \ref{prop:class-S}. 
It is clear that $j_V(V_s)$ is full by minimality of $V_s$. Moreover, it is nondegenerate since it is a submodule of a nondegenerate one. Hence, $\Sigma\circ\Lambda_\cS$ is indeed well-defined satisfying $\Sigma\circ\Lambda_\cS(V)\simeq j_V(V_s).$

We consider the compositions 
$$F:=\Sigma\circ \Lambda_\cS\circ\Pi^{\nd}_{\full} \text{ and } G:=\Pi^{\nd}_{\full}\circ\Sigma\circ \Lambda_\cS.$$
Given $V$ full and nondegenerate, we have that $F(V)$ is isomorphic to $j_V(V_s)$. Since $V$ is nondegenerate $j_V$ is injective and since $V$ is full $V_s$ is equal to $V$. Hence, $F(V)\simeq V$ and this isomorphism is natural.
Conversely, consider $W\in \Rep_\cS(L)$. There exists $V\in\cS$ so that $W\simeq \Pi(V)$.
Then, $\Sigma\circ\Lambda(W)=j_V(V_s)$ which is isomorphic to the quotient module $V_0:=V_s/(\ker(j_V)\cap V_s)$.
Then $\Pi(V_0)\simeq \Pi(V_s)\simeq \Pi(V)$ by Propositions \ref{prop:degenerate-functor} and \ref{prop:complete-rep}.
Hence, $G(W)\simeq W$ and one can check that this defines a natural isomorphism (between the identity functor and $G$).
All together we have proven that $\Pi_{\full}^{\nd}$ and $\Sigma\circ \Lambda$ define an equivalence of categories. This finishes the proof of the Theorem.

For the sake of clarity we give a different proof showing that $\Pi^{\nd}_{\full}$ is fully faithful.
For each pair $(V,V')$ full and nondegenerate we have a mapping 
$$\Pi_{V,V'}:\Hom(V,V')\to \Hom(\Pi(V),\Pi(V')), \theta\mapsto \Pi(\theta).$$
We have the following description:
$$\Pi(\theta) [t,v_1,\dots,v_m]= [t,\theta(v_1),\dots,\theta(v_m)]$$
where $t$ is a tree with $m$ leaves and $v_i$ is the decoration of the $i$th leaf of $t$.
In particular, $\Pi(\theta)[I,v]=[I,\theta(v)]$.
Since $V,V'$ are nondegenerate we have that $j:V\to \Pi(V)$ and $j':V'\to\Pi(V')$ are injective.
We then deduce that $\Pi(\theta)\restriction_{j(V)}=\theta$ up to identifying $V$ with $j(V)$ and $V'$ with $j'(V')$.
This proves that $\Pi_{V,V'}$ is injective.
Consider now a morphism of $L$-modules $\alpha:\Pi(V)\to\Pi(V')$.
Since $V,V'$ are in $\cS$ we have that $\Lambda\circ\Pi(V),\Lambda\circ\Pi(V')$ are in $\cS$ too with smallest submodules $j(V_s),j'(V'_s)$ by (4) of Proposition \ref{prop:class-S}.
By (3) of the same proposition $\alpha$ restricts into a morphism of $A$-modules $\alpha\restriction:j(V)\to j'(V')$.
Since $V,V'$ are full and nondegenerate we have $V=V_s\simeq j(V_s)$ and $V'=V'_s\simeq j'(V'_s)$.
Hence, we deduce a morphism of $A$-modules $\theta:V\to V'$ satisfying that 
$$\alpha \left([I,v]\right)=[I,\theta(v)].$$
Using that $\alpha$ is an $L$-equivariant mapping, it can be shown that $\alpha = \Pi(\theta)$. This proves that $\Pi_{V,V'}$ is surjective, giving all together that $\Pi_{\full}^{\nd}$ is fully faithful.
\end{proof}

\subsection{Irreducibility and indecomposability}

Theorem \ref{theo:main} permits to decide if two representations of $L$ are equivalent or not solely using representations of $A$. 
We now prove $\Pi$ preserves irreducibility and under certain conditions indecomposability. This is not implied by the preceding theorem since we are working in non-semisimple categories where Schur lemma has no converse.
The nerve of the argument is contained in the following elementary lemma. Note, in the analytical situation of the Cuntz algebra this lemma was a complete surprise and much harder to obtain, see \cite[Theorem 4.4]{Brothier-Wijesena24b}.

\begin{lemma}\label{lem:L-submodule}
Consider a representation $W\in \Rep(L)$ that we identify with $\Pi(V)$ for a suitable $V\in \Rep(A)$.
Let $W_0\subset W$ be a nonzero $L$-submodule.
Then we have that $j_V(V)\cap W_0$ is a nonzero $A$-submodule of $j_V(V)$.
\end{lemma}

\begin{proof}
Consider $W\in \Rep(L)$ with $V$ as above. 
Let $W_0\subset W$ be a nonzero $L$-submodule.
By assumption there exists $[t,v]\in W_0$ nonzero. This means that there exists a leaf $\ell$ of $t$ such that the component $v_\ell$ is nonzero and moreover $\Phi(s)(v_\ell)\neq 0$ for any tree $s$.
Letting $p$ be the path corresponding to $\ell$ we obtain $ [t,v]\cdot p=[I,v_\ell]=j_V(v_\ell)$ which is nonzero and which is in $j_V(V)\cap W_0.$
\end{proof}

A representation is $(\pi,V)$ is called nonzero if $\pi(e_i)\neq 0$ for at least one generator $e_i.$

\begin{theorem}\label{theo:irred}
If $V$ is an irreducible (resp.~indecomposable, nondegenerate and full) representation of $A$, then $\Pi(V)$ is an irreducible (resp.~indecomposable) representation of $L$.
Moreover, if $V,V'\in \Rep(A)$ are irreducible and nonzero (resp.~indecomposable, nondegenerate and full), then $V\simeq V'$ if and only if $\Pi(V)\simeq \Pi(V')$.
\end{theorem}

\begin{proof}
Consider an irreducible $V\in\Rep(A)$ and write $W$ for the $L$-module $\Pi(V)$.
If $W=\{0\}$, then it is irreducible.
Assume $W\neq\{0\}$ and fix a nonzero $L$-submodule $W_0\subset W$.
By Lemma \ref{lem:L-submodule} we have that $W_0\cap j_V(V)\subset j_V(V)$ is a nonzero $A$-submodule.
Since $V$ is irreducible so is $j_V(V)$ and thus $W_0\cap j_V(V)=j_V(V)$.
Since $j_V(V)\subset W$ is complete we deduce that $W_0=W$, thus $W$ is irreducible.

Assume now that $V\in\Rep(A)$ is indecomposable, nondegenerate and full.
We set $W:=\Pi(V)$ and assume that it decomposes as an $L$-module into $W=W_1\oplus W_2$.
Since $V$ is nondegenerate we may identify $V$ with its image $j_V(V)$ inside $W$.
Observe that $V_i:=V\cap W_i$ generates $W_i$ as an $L$-module for $i=1,2$.
Indeed, given $[t,v]\in W_i$ we have $V_i\ni [I,v_\ell]= [t,v]\cdot e_\ell$ for all leaf $\ell$ of $t$ and thus $[t,v]=\sum_{\ell} [I,v_\ell]\cdot e_\ell^*$.
Since $W_1\cap W_2=\{0\}$ we obtain that the $V_i$ are in direct sum inside $V$. 
Moreover, $V_1\oplus V_2$ is complete since it generates $W_1\oplus W_2=W$.
Fullness of $V$ implies the equality $V=V_1\oplus V_2$.
Since $V$ is indecomposable we deduce that one of the $V_i$ is trivial, say $V_2$.
Lemma \ref{lem:L-submodule} implies that $W_2=\{0\}$ and thus $W$ is indecomposable.

The second statement is trivially implied by Theorem \ref{theo:main} and Proposition \ref{prop:degenerate-rep}.
\end{proof}

The proof of the Theorem of above implies the following.

\begin{corollary}
For $V\in\Rep(A)$ the $L$-module $\Pi(V)$ is irreducible if and only if all $A$-submodules of $V$ are either complete or contained in $\ker(j_V)$.
\end{corollary}

\begin{remark}
Any $L$-module $W$ is isomorphic to $\Pi(V)$ for some $V\in\Rep(A)$. 
Moreover, $W \in \Rep_{\cS}(L)$ if and only if $V \in \cS$.
Additionally, if $V$ is irreducible, then $\Pi(V)$ is irreducible.
Finally, if $W\in \Rep_\cS(L)$ is irreducible, then $V:=\Sigma(\Lambda(W))$ is an irreducible $A$-module satisfying $\Pi(V)\simeq W.$
Though, this does not mean that all irreducible $L$-modules can be constructed from an irreducible $A$-module.
Indeed, consider the irreducible Chen module $F_p$ for some irrational ray $p$ (see Section \ref{sec:Chen-non-S}).
We have constructed a reducible $A$-module $C_p$ satisfying $\Pi(C_p)\simeq F_p$ and shown that $C_p\notin \cS$, implying $F_p \notin \Rep_{\cS}(L)$. 
Thus, we deduce that if $V$ is an irreducible $A$-module, then $\Pi(V)\not\simeq F_p$. A similar argument holds when $F_p$ is replaced by any other irreducible $L$-module not in $\Rep_\cS(L).$
\end{remark}

\subsection{A novel dimension function}
We now define two dimension-functions: one for representation of $A$ and one for representations of $L$. 
They are both denoted by $\dim_A$, calling them \emph{$A$-dimensions}. They are valued in $\N\cup\{\infty,\varepsilon\}$ where $\varepsilon$ is a fixed symbol.
In what follows we write $\dim$ for the usual dimension of a $\bk$-vector space.
\begin{definition}
Consider $V\in \Rep(A)$ and $W\in \Rep(L)$.
We set
\begin{itemize}
\item $\dim_A(V)=\varepsilon$ when $V\notin\cS$ and $\dim(\nabla\circ\Sigma(V))$ otherwise;
\item $\dim_A(W)=\varepsilon$ if $W\not\simeq\Pi(V')$ for all $V'\in\cS$ and is otherwise equal to the infimum of $\dim(V')$ over all $V'\in\Rep(A)$ satisfying $\Pi(V')\simeq W$.
\end{itemize}
\end{definition}

In particular, $\dim_A$ restricted to $\Rep_\cS(A)$ is equal to the composition of functors $\dim\circ \nabla\circ \Sigma$ and thus is well-defined, invariant under isomorphisms, and additive.
We obtain a similar statement for the dimension function $\dim_A$ of $L$-modules since $\dim$ and $\Pi$ are functors. 
Note that the infimum giving $\dim_A(W)$ can be taken over all $A$-submodules of $\Lambda(W)$, making it an infimum on a \emph{set} rather than on a \emph{collection}.
We now relate the two dimension-functions. In particular, we give a very practical way for computing the $A$-dimension of an $L$-module.

\begin{theorem} \label{thm:l-dim}
We have the following equalities:
\begin{enumerate}
\item $\dim_A(V)=\dim_A(\Pi(V)) \text{ for all } V\in \Rep(A)$; and
\item $\dim_A(W)=\dim_A(\Lambda(W)) \text{ for all } W\in \Rep(L).$
\end{enumerate}
Additionally, $\dim_A:\Rep_\cS(A)\to\N\cup\{\infty\}$ is invariant under the endofunctors $\Sigma, \nabla, \Lambda\circ\Pi$ and $\dim_A:\Rep(L)\to \N\cup\{\infty,\varepsilon\}$ is invariant under $\Pi\circ \Lambda$.
In particular, if $V$ is full and nondegenerate, then $\dim_A(\Pi(V))$ is equal to the \emph{usual} dimension of $V$.
\end{theorem}

\begin{proof}
Consider $V\in\Rep_\cS(A)$.
From Section \ref{subsec:full-nondeg} we have that $\nabla\circ\Sigma(V)\simeq \Sigma\circ\nabla(V)\simeq \Sigma\circ\Lambda\circ\Pi(V).$
This implies that $\dim_A:\Rep_\cS(A)\to\N\cup\{\infty\}$ is invariant under $\Sigma,\nabla$ and $\Lambda\circ\Pi$.
Moreover, we note that $\dim(\nabla\circ\Sigma(V))\leq \dim(V).$
Hence, when computing $\dim_A(W)$ it is sufficient to take the infimum only over full nondegenerate modules.

Fix $V\in \Rep_\cS(A)$ and let $V_1\in \Rep(A)$ be full nondegenerate satisfying $\Pi(V)\simeq \Pi(V_1).$
Note that $V_2:=\Sigma\circ\nabla(V)$ is full nondegenerate and $\Pi(V_2)\simeq\Pi(V_1)$.
Theorem \ref{theo:main} implies $V_2\simeq V_1$ and thus $\dim(V_2)=\dim(V_1).$
By the observation above we deduce that $\dim_A(\Pi(V_2))=\dim(V_2).$
Since $V_2$ is full and nondegenerate we have $\dim(V_2)=\dim_A(V_2)$ and since $\dim_A=\dim_A\circ \Sigma\circ\nabla$ we deduce $\dim_A(V_2)=\dim_A(V)$.
This proves the first statement:
$\dim_A(V)=\dim_A(\Pi(V)).$

Consider $W\in \Rep_\cS(L)$.
We have seen that $\Pi\circ\Lambda(W)\simeq W$ as $L$-modules.
Hence, the first statement applied to $\Lambda(W)$ gives $\dim_A(\Lambda(W))=\dim_A(\Pi\circ\Lambda(W))=\dim_A(W).$
\end{proof}

\subsection{Moduli spaces} \label{subsec:geom-irrep}

Let $\Irr(L)$ be the collection of irreducible classes of representation of $L$.
Using the dimension function $\dim_A$ we partition $\Irr(L)$ into $\sqcup_d \Irr(L)_d$ for $d\in\N\cup\{\infty,\varepsilon\}$ with $\Irr(L)_d$ being the irreducible classes of $A$-dimension $d$.
We now explain how for each finite $d$ the set $\Irr(L)_d$ is in bijection with a smooth variety, thus obtaining moduli spaces. 
To achieve this we make the following assumption on the field.

\begin{center}
{\bf In this section the field $\bk$ is algebraically closed and of characteristic zero.}
\end{center}

Given $d$ finite we write $\Irr(A)_d$ for the irreducible classes of representation of $A$ of \emph{usual} dimension $d$. 
Now, $\Pi$ is fully faithful and preserves irreducibility on full and nondegenerate representations.
Moreover, all classes in $\Irr(A)_d$ are full and nondegenerate except for one class $[\pi]$: when $d=1$ and $\pi(e)=0$ for all generator $e$.
We deduce that $\Irr(L)_d$ is in bijection with $\Irr(A)_d$ for $d\neq 1$ finite and that $\Irr(L)_1$ is in bijection with $\Irr(A)_1$ minus one point.
It is thus sufficient to find moduli spaces for the $\Irr(A)_d.$

For $d=1$, $\pi\mapsto (\pi(e_1),\dots,\pi(e_n))$ provides a bijection from $\Irr(A)_1$ to $\bk^n$ and then $\Irr(L)_1$ is in bijection with $\bk^n\setminus\{0\}$.
Fix $d\geq 2$. 
We will use well-established result from geometric invariant theory to describe $\Irr(A)_d$ \cite{Mumford-Fogarty-Kirwan94} and follow King's approach \cite{King94}, see also \cite{Reineke08}.
Take $R_d$ the vector space of all representations $\pi:A\to M_d(\bk)$ which is isomorphic to $M_d(\bk)^n$. 
The formula $g\cdot (b_1,\dots,b_n):=(gb_1g^{-1},\dots,gb_ng^{-1})$ defines an action $PG_d\act R_d$ of the projective linear group $PG_d:=GL_d(\bk)/\bk^\times$.
The \emph{stable} representations of $R_d$ are the $\pi$ with (Zariski) closed orbit $PG_d\cdot \pi$ and finite stabiliser subgroup.
They form an open subset $R_d^{st}.$
One can prove that they are exactly the irreducible ones (closed orbit is equivalent to semisimplicity and finite stabiliser gives simplicity).
Better, the restricted action $PG_d\act R_d^{st}$ is \emph{free}.
Indeed, if $g\in GL_d(\bk)$ then it admits an eigenvector (since $\bk$ is algebraically closed) with eigenspace $E$. If $g\cdot b=b$, then each $b_j$ stabilises $E$.
Irreducibility implies $E=\bk^d$ and thus $g$ is trivial inside the quotient $PG_d$ proving freeness.
The orbit space $M^{st}_d:=R_d^{st}/PG_d$ of the free and algebraic action $PG_d\act R_d^{st}$ of the reductive group $PG_d$ inherits a structure of smooth variety of dimension:
$\dim(M^{st})=\dim(R_d^{st})-\dim(PG_d) = nd^2-(d^2-1)=(n-1)d^2+1.$

\begin{theorem}\label{theo:moduli-space}
The space of irreducible classes of representations of $L$ of fixed $A$-dimension $d\in\N^*$ is in bijection with a smooth variety of dimension $(n-1)d^2+1$.
\end{theorem}

When $\bk = \C$ we may define finer geometric structure, see \cite{Nakajima94} and \cite[Section 6]{King94}.
In this specific case of the free algebra there are other ways to define moduli spaces for $\Irr(A)_d$; for instance the interested reader may consult Artin's lecture notes \cite{Artin99}.

\section{Extending to Quiver and Leavitt path algebras} \label{sec:path-algebra}

This section explains how to extend the previous results to Leavitt \emph{path} algebras associated to \emph{arbitrary row-finite} directed graphs.
We will fully describe the main constructions and point out the main differences.
For details on Leavitt path algebras we recommend the book \cite{Abrams-Ara-Molina17} and the recent survey \cite{Abrams-Hazrat24}.

\subsection{Graphs and Leavitt path algebras.}
A directed graph $\Gamma$ (also called a quiver) consists of a quadruple $(E^0, E^1, r, s)$ where $E^0$ is the set of vertices, $E^1$ is the set of (directed) edges (or arrows), and $r,s$ are maps $E^1 \rightarrow E^0$. For each edge $e \in E^1$, $r(e)=re$ is the \textit{range} of $e$ and $s(e)=se$ is the \textit{source} of $e$. 
A vertex $\nu$ is a \emph{sink} when $E^1_\nu:=\{e\in E^1:\ se=\nu\}$ is finite, an \emph{infinite emitter} when $|E^1_\nu|=\infty$, and otherwise is \emph{regular}.
The graph is \emph{row-finite} when there are no infinite emitters.

\textbf{Paths in a graph.}
A \textit{path} of length $n > 1$ is a sequence of edges $p := e_1e_2\dots e_n$ such that $re_i = se_{i+1}$ for $i = 1, \dots, n-1$. 
We extend $s,r$ to paths so that $sp=se_1$ and $rp=re_n$ and say that $p$ has length $n$. 
Lengths $0$ and $1$ paths are vertices and edges, respectively.
A \emph{cycle} is a path with same range and source, and a \emph{loop} is a cycle of length $1$.
Write $E^n$ for the set of paths of length $n$ or end in a sink having length $k\leq n$. 
Let $E^*$ be the set of all paths which forms a category in the obvious way. 
Formally, a path $p$ is a morphism of this category with domain $rp$ and codomain $sp.$

\textbf{The Leavitt path algebra.}
As in the previous section we fix a field $\bk$ and all vector spaces and algebras are over it.
For a directed graph $\Gamma$, its \textit{extended graph} $\ti \Gamma$ is obtained from $\Gamma$ by adding inverses $e^*$ of all edges $e$ of $\Gamma$.
We extend $^*$ to an involution on all paths of $\ti \Gamma$ by setting $\nu^* := \nu$ for vertices and $(e_1\dots e_n)^* := e_n^*\dots e_1^*$.

\begin{definition} \label{def:LPA}
The quiver algebra $A$ of a directed graph $\Ga$ is generated by $E^0\sqcup E^1$ satisfying the relations:
\begin{align*}
&\nu\omega = \delta_{\nu,\omega} \textrm{ for all } \nu,\omega \in E^0, \tag{$E0$}\\
&s(e)e = e = er(e) \textrm{ for all } e \in E^1, \tag{$E1$}. 
\end{align*}

The \emph{Leavitt path algebra} $L$ of $\Gamma$ is the quotient of the quiver algebra of the \emph{extended} graph $\tilde \Ga$ (hence adding the reverse edges $e^*$ as generators) by the relations:
	\begin{align}
		&e^*f = \delta_{e,f}r(e) \textrm{ for all } e,f \in E^1, \tag{CK1} \\
		&\nu = \sum_{e \in E_\nu^1} ee^* \textrm{ for all regular vertices } \nu. \tag{CK2}
	\end{align}
\end{definition}

\begin{example}
	\begin{itemize}
		\item The Leavitt path algebra of a bouquet of $n$ loops is the usual $n$th Leavitt algebra studied in the previous sections. 
		\item The $n$-line graph consists of $n$ vertices $\{\nu_i\}_{i=1}^n$ and $n-1$ edges $\{e_j\}_{j=1}^{n-1}$ such that $se_i = \nu_i$ and $re_i = \nu_{i+1}$. The Leavitt path algebra is the matrix algebra $M_n(\bk)$.
		\item The $n$-circle graph is formed by adding an edge from $\nu_n$ to $\nu_1$ in the $n$-line graph. The Leavitt path algebra is the matrix algebra $M_n(\bk[x,x^{-1}])$ where $\bk[x,x^{-1}]$ is the ring of Laurent polynomials.
	\end{itemize}
\end{example}

\subsection{Representations of path algebras and Leavitt path algebras}

\begin{center}
From now on $\Ga$ is a fixed {\bf row-finite} directed graph and $A,L$ are its associated quiver algebra and Leavitt path algebra.
\end{center}

\subsubsection{Representations of path algebras}

A \emph{quiver representation} $(\pi,V)$ is a contravariant functor from the path category of $\Ga$ to $\Vect$. Hence, $(\pi,V)$ defines a family of vector space $(V_\nu:\ \nu\in E^0)$, and each edge $e\in E^1$ defines a linear map $\pi(e):V_{se}\to V_{re}.$
We identify quiver representations with right-modules of $A$ that we also call representations via $(V_\nu)_\nu\mapsto \oplus_\nu V_\nu$ and $\pi(e)\restriction_{V_\omega}=0$ if $\omega\neq se.$ 
These representations are exactly the \emph{unital} one (in the weak sense: $\textrm{span}\{v\cdot a:\ v\in V, a\in A\}=V$). 
They form a category $\Rep(A)$. 

The \emph{dimension vector} $\dim_\Gamma(V)$ is the tuple $(\dim(V_\nu):\ \nu \in E^0)$, i.e.~$\dim_\Ga:E^0\to \N\cup\{\infty\}$ is a function with domain $E^0$.

We define $\Rep(L)$ to be the category of unital representations of $L$.
There is an obvious algebra morphism $\iota:A\to L, e_i\mapsto e_i$ given by universality of $A$. This defines a forgetful functor $\Lambda:\Rep(L)\to\Rep(A)$.

\subsubsection{Trees and forests}
As in the base case we can define a certain category of forests $\cF$ and a certain isomorphism with $\Func_\ot(\cF,\Vect)$. 
This is not essential to our discussion but permits to interpret our construction via Jones' technology and to use convenient notations. We shall be rather brief.

For each vertex $\nu$ we consider its associated $\nu$-caret $\Y_\nu$. This is a rooted full tree with all leaves at distance one so that the root is labelled by $\nu$, and any edge $e\in E^1$ with $se=\nu$ defines a leaf labelled by $re.$ 
Trees are obtained by gluing together finitely many carets lining up roots and leaves with same labelling. Forests are finite lists of trees.
Now, for each $\nu\in E^0$ we fix an arbitrary order on 
$$E^1_\nu:=\{e\in E^1:\ se=\nu\}.$$
This permits to order leaves and roots of forests. 
We have a structure of a monoidal category $(\cF,\circ,\otimes)$ obtained by stacking and by concatenating.
Objects of $\cF$ are finite lists $\nu_1\dots\nu_n$ of vertices of $\Ga$ and $\Hom(\nu_1\dots\nu_n,\mu_1\dots\mu_m)$ is the set of forests with $m$ roots, $n$ leaves, and matching labels.
We previously wrote $I$ for the trivial tree. 
Now, each $\nu\in E^0$ produces one trivial tree, still denoted $\nu$, that has no edges and one vertex labelled by $\nu$.
A presentation of this category is obtained by considering elementary forests (forests with one caret) and writing the two possible decompositions of a forest with two carets rooted at different roots. 
We then deduce an isomorphism of categories,
$$\Caret:\Rep(A)\to \Func_{\otimes}(\cF,\Vect)$$
where given $(\pi,V)\in \Rep(A)$ we define for each $\nu\in E^0$ the caret map 
$$V_\nu\to \bigoplus_{e\in E^1_\nu} V_{re} \ , \ v\mapsto \bigoplus_{e\in E^1_\nu}v\cdot e.$$ 

\begin{remark}
If $\nu$ would be an infinite emitter, then its caret map would involve infinitely many factors. This does not work well in this algebraic context and we did not find any satisfactory way to resolve this problem.
\end{remark}

\subsubsection{Functors between $\Rep(A)$ and $\Rep(L)$.}
Fix $(\pi,V)\in \Rep(A)$ and define
\[\Pi(V) := \textrm{span}\{(p, x) : p \in E^*,\ x \in V_{rp}\}/\sim\]
where the equivalence relation $\sim$ is generated by
\begin{equation}\label{eq:sim}(p, x) \sim \sum_{e\in E^1_{rp}} (pe, x\cdot e).\end{equation}
Recall that $E^1_{rp}$ is the set of edge with source $rp$. 
By assumption $\Ga$ is row-finite implying $|E^1_{rp}|<\infty$ and the formula of above is well-defined.
Write $[p,x]$ for the equivalence class of $(p,x)$. 
Hence, elements of $\Pi(V)$ are linear combinations of classes of paths $p$ decorated by vectors of $V_{rp}$.
The space $\Pi(V)$ decomposes into a direct sum $\oplus_{\nu\in E^0}\Pi(V)_\nu$ where $\Pi(V)_\nu$ is the subspace of $[p,x]$ with $sp=\nu$.

We define an $L$-module structure on $\Pi(V).$
Consider any pair of paths $p,q\in E^*$ (possibly equal to vertices) and $x\in V_{rp}.$
We set 
$$[p,x]\cdot q^*:=
\begin{cases}
    [qp,x] &\text{ if } rq=sp \\
    0 &\text{ otherwise.}
\end{cases}.$$
Thus, if $q=\nu\in E^0$ then $[p,x]\cdot \nu=\delta_{sp,\nu}[p,x]$: the coordinate projection onto $\Pi(V)_\nu$. 

Consider $e,f\in E^1, u\in E^*, x\in V_{ru}$ and set:
$$[eu,x]\cdot f := \delta_{f,e} [u,x].$$
Replace now the path $eu$ by a vertex $\nu$.
If $\nu$ is not a sink, then 
$$[\nu,x] \cdot f := \sum_{e\in E_\nu^1} [e,x\cdot e] \cdot f =\begin{cases}
[rf,x\cdot f] & \text{ if } sf=\nu\\
0 & \text{ otherwise.}
\end{cases}.$$
Assume now that $\nu$ is a sink. We set $[\nu,x]\cdot f=0$.
This is coherent with the previous formula since $f=sf\cdot f$ and that $[\nu,x]\cdot sf=0$ since $\nu\neq sf.$

The above defines an $L$-module structure. 
Indeed, $\nu\in E^0$ defines the coordinate projection $\Pi(V)\onto\Pi(V)_\nu$ and $e\in E^1$ (resp.~$e^*$) a map from $\Pi(V)_{se}\to\Pi(V)_{re}$ (resp.~from $\Pi(V)_{re}\to\Pi(V)_{se}$) giving (E0) and (E1).
Formula (CK1) and (CK2) are easily verified noting that (CK2) only hold for \emph{regular} vertices (hence excluding sinks).

Consider an intertwinner $\theta:V\to V'$ between representations of $A$. 
This naturally defines an intertwinner of $L$-modules $\Pi(\theta):\Pi(V)\to\Pi(V')$ via the formula
\[\Pi(\theta)([p, x]) := [p, \theta_{rp}(x)].\]
Hence, we have constructed a functor $\Pi : \Rep(A) \rightarrow \Rep(L)$.

\begin{remark}In the base case (when $|E^0|=1$) we described vectors in $\Pi(V)$ using (classes of) trees $[t,v]$ with leaves decorated by vectors of $V$. 
We may proceed similarly but we will still need to do linear combinations (when $|E^0|\geq 2$) to sum decorated trees having roots of different labels. This is because the tree set $\cT$ is no longer directed but decomposed into $\sqcup_{\nu\in E^0}\cT_\nu$ where each $\cT_\nu$ (trees with roots $\nu$) is now directed.
Elements of $\Pi(V)$ are sums of $[t_\nu,v_\nu]$ with $t_\nu\in\cT_\nu$.
For a fixed $\nu\in E^0$, the set of $[t_\nu,v_\nu]$ is the $\nu$-component $\Pi(V)_\nu$.
\end{remark}

We have a \emph{canonical map} $j := j_V: V \rightarrow \Pi(V)$ which is a morphism of $A$-modules that now decomposes over $E^0$ into $j=(j_\nu)_{\nu\in E^0}$ with 
$$j_\nu: V_\nu\to \Pi(V)_\nu,\ x\mapsto [\nu,x].$$
Moreover, $(j_V:\ V\in \Rep(A))$ defines a natural transformation $\id\to \Lambda\circ\Pi.$
Note that when $\nu$ is a sink, then the $\nu$-component of $j$ is always injective.
Proposition \ref{prop:jones-functor-surj} extends in the obvious manner so that $(j_W)_W$ with $W$ in the range of $\Lambda$ defines a natural isomorphism $\id\to\Pi\circ \Lambda$ of endofunctors of $\Rep(L)$.
The key point here consists in showing that if $W\in \Rep(L)$, then $j_W:W\to \Pi(W), x\mapsto \sum_{\nu\in E^0} [\nu, x\cdot \nu]$ is surjective.
Indeed, take $[p,w]\in \Pi(W)$ for some path $p$ of length $n$ and $w\in W_{rp}.$
Then observe
$$j_W(w\cdot p^*) = [sp, w\cdot p^*] = \sum_{q \in E^n_{sp}} [q, w\cdot p^*q] = [p,w]$$
where $E^n_{sp}$ is the subset of $E^n$ containing paths with source $sp$.

\subsection{Complete, full and degenerate representations}

Definitions and results extend in the obvious manner.

\begin{definition}
Let $V$ be a representation of $A$.
\begin{enumerate}
\item A subrepresentation $V_c \subset V$ is \emph{complete} if for all $v \in V$ there exists $k \geq 1$ satisfying $v\cdot p\in V_c$ for all $p\in E^k.$       
\item $V$ is \emph{full} if it does not contain any proper complete subrepresentations.
\item If $j_V:V\to\Pi(V)$ is injective, then $V$ is \emph{nondegenerate}, otherwise \emph{degenerate}.
\item Write $\Comp(V)$, $\cS$, $\Rep_S(A)$, $\Rep_{\full}(A)$, $\Rep^{\nd}(A)$ as defined in Section \ref{sec:base-case}.
\end{enumerate} 
\end{definition}

As before, $V_c \subset V$ is complete exactly when the inclusion map $\epsilon: V_c \hookrightarrow V$ defines an isomorphism $\Pi(\epsilon) : \Pi(V_c) \rightarrow \Pi(V)$ of $L$-modules which is equivalent to ask that $j(V_c)$ generates the $L$-module $\Pi(V).$
Recall that here $E^k$ include all paths of length smaller than $k$ ending in a sink and observe that if $V_c\subset V$ is complete, then $V_c$ contains $V_\nu$ when $\nu$ is a sink.
Proofs of Propositions \ref{prop:complete-rep} generalise by using paths in $E^k$ rather than product of $k$ generators.
To generalise Proposition \ref{prop:degenerate} we replace the caret map $\Phi(\Y)$ by the direct sum of all labelled caret maps $C:=\oplus_{\nu\in E^0} \Phi(\Y_\nu)$. The generalised key statement is now: $j_V$ is injective if and only if $C$ is.
Proposition \ref{prop:class-S} on class $\cS$ continues to hold with quasi-identical proofs. 
For precision, below is the analog of item (5) of this proposition.
\begin{proposition}
If $V\in \Rep(A)$ and $\dim(V)=\sum_{\nu\in E^0}\dim(V_\nu)$ is finite, then $V\in \cS$.
\end{proposition}
We cannot hope that $\dim(V_\nu)<\infty$ for all $\nu\in E^0$ gives $V\in \cS$ as demonstrated below.

\begin{remark}
Consider $\Ga$ with $E^0=\N$ and one edge $e_n:n\to n+1$ per $n\in \N.$
Define the $A$-module $V$ having basis $(b_n)_n$ with $\bk b_n=V\cdot n$ so that $b_n\cdot m=\delta_{n,m} b_{n+1}.$
For each $k\in\N$ the $b_j$ with $j\geq k$ span a complete submodule $Z_k$. Since $\cap_k Z_k=\{0\}$ we have $V\notin\cS$.
\end{remark}

As for the Leavitt algebra, we are able to functorially both extract the smallest complete submodule and remove degenerescence. 
This gives two functors 
$$\Sigma:\Rep_\cS(A)\to \Rep_{\full}(A) \text{ and }\nabla:\Rep(A)\to \Rep^{\nd}(A)$$ 
with again $\nabla(V):=V/\ker(j_V)$.
There are natural isomorphisms between the three endofunctors of $\Rep_{\cS}(A)$:
		\[\nabla \circ \Sigma, \Sigma \circ \nabla \textrm{ and } \Sigma \circ \Lambda \circ \Pi.\]
We generalise Proposition \ref{prop:degenerate-rep} noting that each regular vertex of $\Ga$ gives one degenerate irreducible representation of $A$.

\begin{proposition} \label{prop:irred-deg-path}
An irreducible representation $(\pi, V)$ of $A$ is degenerate when $V=V_\nu$ is one-dimensional for a regular vertex $\nu$ and $\pi(e)=0$ for all $e\in E^1.$
\end{proposition}

\begin{proof}
Consider an irreducible degenerate representation $(\pi, V)$ of $A$. 
We necessarily have $\ker(j_V)=V$ implying $\pi(e)=0$ for all edge $e$ which then implies that any vector subspace is a submodule.
By irreducibility $V=V_\nu\simeq\bk$ for a certain $\nu$ and this $\nu$ is not a sink since otherwise $j_V$ would be injective.
\end{proof}

Proofs of Theorems \ref{theo:main}, \ref{theo:irred} now adapt in the obvious way giving items (1) and (2) of Theorem \ref{letterthm:main-graph}.

\subsection{Dimension}
As for the Leavitt algebra, we can define two dimension functions on $A$-modules and on $L$-modules both denoted $\dim_A$. They are now valued in $(\N \cup \{\infty\})^{E^0} \cup \{\varepsilon\}$ giving $A$-dimension \emph{vectors}. Recall, for a representation $V$ of $A$ we write $\dim_\Gamma(V)$ for the dimension vector $(\dim(V_\nu))_{\nu\in E^0}$.
We order dimension vectors as follows: $d\leq \widehat d$ when for all $\nu\in E^0$ we have $d_\nu\leq \widehat d_\nu.$

\begin{definition}
Consider $V\in \Rep(A)$ and $W\in \Rep(L)$.
We set
\begin{itemize}
\item $\dim_A(V)=\varepsilon$ when $V\notin\cS$ and $\dim(\nabla\circ\Sigma(V))$ otherwise;
\item $\dim_A(W)=\varepsilon$ if $W\not\simeq\Pi(V')$ for all $V'\in\cS$ and is otherwise equal to the infimum of $\dim_\Ga(V')$ over all $V'\in\Rep(A)$ satisfying $\Pi(V')\simeq W$.
\end{itemize}
\end{definition}
Note, the above infimum exists since if $V_1,V_2 \in \Rep(A)$ satisfy $\Pi(V_1) \simeq \Pi(V_2) \simeq W$, then after identifying $\Pi(V_1)$ with $\Pi(V_2)$ we note that $V := j_{V_1}(V_1) \cap j_{V_2}(V_2)$ satisfies $\Pi(V) \simeq W$ and has smaller dimension vector.
Theorem \ref{thm:l-dim} extends in the obvious manner; in particular $\dim_A(\Pi(V))=\dim_A(V)$ and $\dim_A(\Lambda(W))=\dim_A(W).$

\subsection{Moduli spaces}\label{sec:moduli-space-path}
For this section we assume that $\bk$ is algebraically closed  of characteristic $0$, and continue to use similar notation for $\Irr(L), \Irr(L)_d, \Irr(A), \Irr(A)_d$, where now $d$ refers to a dimension vector.
Recall that the $d$ of $\Irr(L)_d$ refers to the \emph{$A$-dimension} while the $d$ of $\Irr(A)_d$ refers to the \emph{usual} dimension.

Consider $(\pi,V)\in \Rep(A)$ irreducible with $\dim(V)=\sum_{\nu\in E^0} \dim(V_\nu)<\infty$ .
We have that $V\in \cS$.
Moreover, $V$ is full, nondegenerate, and $\dim_\Ga(V)=\dim_A(V)$ unless $V=V_\mu\simeq\bk$ for a certain regular vertex $\mu$ and $\pi(e)=0$ for all generators.
In this latter case $\dim_A(V)$ is the zero vector.
Hence, the irreducible classes of representations of $A$ of $A$-dimension $d$ and $\Irr(A)_d$ differs at most by one point.
Our main theorems prove that the functor $\Pi$ preserves the $A$-dimension, the irreducibility, and moreover the classes when the $V$'s are nondegenerate and full.
Hence, up to a point, $\Irr(A)_d$ and $\Irr(L)_d$ are in bijection.
It is then sufficient to find moduli spaces for the former.
We briefly outline this last step which follows from  \cite{King94}, see also \cite{Reineke08}.
Finer geometric structures are available when $\bk=\C$ (see \cite{Nakajima94, King94}) and interesting results remain when $\bk=\R$ \cite{Bertozzi-thesis}.

Fix a summable $d$ and take $R_d$ to be the vector space of all $V\in \Rep(A)$ with $\dim_\Ga(V)=d.$
This is parameterised by the product of matrix algebras of size $d_{re}\times d_{se}$ indexed by $e\in E^1$.
Consider the linear group
$$G_d = \prod_{\nu \in E^0} GL(d_\nu,\bk)$$
and the natural base change action $PG_d:=G_d/\bk^\times\act R_d$. 
The \textit{stable} representations are precisely the irreducible ones which  form an open subset $R_d^{st} \subset R_d$, and $PG_d\act R_d^{st}$ is an algebraic and free action of a reductive group.
Note that $R_d^{st}$ could be empty (e.g.~take the graph $\nu\to\mu$ and $d=(d_\nu,d_\mu)=(0,2)$), see also Section 5.3 of \cite{Reineke08}. 
The quotient $M_d^{st} := R_d^{st}/PG_d$, when nonempty, has a smooth variety structure of dimension $\dim(R_d)-\dim(PG_d)$ and is in bijection with $\Irr(A)_d$.
This yields the third item of Theorem \ref{letterthm:main-graph}.

\section{Simple modules in the literature}

We interpret in our framework previously known simple modules and provide generalisations to obtain new simple modules.

\subsection{Chen modules}\label{sec:Chen}
\subsubsection{Construction of Chen modules.}
Chen modules form the first family of simple $L$-modules defined for an arbitrary graph $\Gamma$, and most other constructions in the literature are generalisations of them \cite{Chen15}.
They are constructed using paths and rays in $\Gamma$. A ray $p$ is \textit{cyclic} if there exists a prime cycle $c$ (there does not exist a cycle $d$ and $k > 1$ such that $c = d^k$) such that $p = c\cdot c\dots=:c^\infty$, and is \textit{irrational} if it is not tail-equivalent to a cyclic ray (two rays are tail-equivalent if they are equal up to finite prefix-removal). 
Chen constructed three families of modules denoted
$$F_q^\lambda, F_p \text{ and } N_\omega$$
where $q$ is a cyclic ray, $\lambda\in \bk^\times$, $p$ is an irrational ray, and $\omega$ is a sink.
If $q = c^\infty$ for some prime cycle $c$, then $F_q^\lambda$ is the free vector space with basis $[q]$ - the tail-equivalence class of $q$. 
Note that $\dim(F_q^\lambda)=\infty$
if and only if there exists a vertex of $c$ contained in another prime cycle. 
Define scalars $(\lambda_r)_{r \in [q]}$ by $\lambda_q := \lambda$ and $\lambda_r := 1$ if $r \neq q$.
The $L$-module structure is:
\[q \cdot f := \lambda_q f^*q \textrm{ and } q \cdot f^* := fq\]
for all $q \in [p], f \in E^1$ taking the convention that $eq=0$ if $re\neq sq$ and $e^*e = r(e)$.
Our description of $F_q^\lambda$ slightly differs from the original one; however, they contain all the original classes of $L$-modules while also permitting to define $F_q^\lambda$ even if $\bk$ does not contain all $n$th roots of $\lambda.$
If $p$ is irrational, then $F_p$ is defined similarly as $F_q^\lambda$ except now $\lambda_r = 1$ for \textit{all} $r \in [p]$ (changing the scalars $\lambda_r$ does not change the isomorphism class of the $L$-module). 
Finally, $N_\omega$ is the free vector space with basis paths ending in $\omega$. The $L$-module structure is similarly defined.

Chen proved that all of the above $L$-modules are simple with the $F_q^\lambda,F_p,N_\omega$ classified by $([q],\lambda), [p],\omega$, respectively.

\subsubsection{Chen modules from $A$-modules.}
It is easy to recover Chen modules using the functor $\Pi$ we have constructed. From our previous results, it is sufficient to find complete $A$-submodules of the Chen modules.

First consider a ray $p$ that is either cyclic or irrational. Let ${}_np$ denote the tail of $p$ obtained by removing the first $n$ edges of $p$ for $n \in \N$. 
Define $T_p=\{p,{}_1p, {}_2p,\dots\}$ with $|c|$ elements when $p=c^\infty$ with $c$ prime,
otherwise being infinite, and write $T_p^\nu\subset T_p$ for those with source $\nu$.
The $A$-submodule $C_p := \bk[T_p]$ (of $\Lambda(F_p^\lambda)$ or $\Lambda(F_p)$) is complete (in both cyclic and irrational case) and satisfies $C_p \cdot \nu = \bk[T_p^\nu]$.
In the cyclic case, $\dim(C_p)=|c|$.

Consider a sink $\omega$. 
Define $C_\omega=\bk \omega$ to be the $A$-submodule generated by the vertex $\omega$ inside $\Lambda(N_\omega)$.
Hence, $C_\omega \cdot \nu = \{0\}$ if $\nu \neq \omega$ and $C_\omega \cdot e = \{0\}$ for all edges $e \in E^1$. Moreover, $C_\omega$ is a complete one-dimensional $A$-submodule.

\subsubsection{Classification via $A$-modules and $A$-dimension vector of Chen modules.}\label{sec:Chen-non-S}
When $p$ is cyclic it is clear the $A$-module $C_p$ is simple and is moreover nondegenerate since it lives inside $\Lambda(F_p^\lambda)$. This assures it is full.
Similarly, $C_\omega$ is simple, nondegenerate and full (compare this to Proposition \ref{prop:irred-deg-path}, the key point being that $\omega$ is a \emph{sink}).
Applying Theorem \ref{letterthm:main-graph} 
we recover simplicity of the Chen modules and classifications of $F_q^\lambda$, $N_\omega$.

If $p$ is irrational then the situation is drastically different. 
Indeed, for each $k\geq 0$ consider $C^k_p\subset C_p$: the span of the $_jp$ with $j\geq k$.
They form a decreasing sequence of complete submodules with trivial intersection. 
This prevents the existence of a smallest complete submodule. 
Hence, $C_p\notin\cS$ and is reducible.
This shows in particular that $C_p$ does not admit a smallest complete submodule, i.e.~$C_p\notin \cS$. 

However, consider an arbitrary vector $v = \sum_{i=1}^n a_iq_i \in C_p$ with $a_i \in \bk$ and $q_i = {}_{n_i}p$ with $n_i \neq n_j$ if $i\neq j$. Since $p$ is irrational, there exists a path $r$ such that it is a prefix of $q_1$ but not one for $q_j$, $j \neq 1$. Thus, $v \cdot r = a_1{}_n(q_1)$ where $n$ is the length of $r$. From this we can deduce that any nontrivial $A$-submodule of $C_p$ is complete. This reproves that $\Pi(C_p)\simeq F_p$ is simple.
By a similar argument to the cyclic case, it is easy to show that the equivalence classes of the $A$-module $C_p \subset \Lambda(F_p)$ is determined by $[p]$. 

Finally, observe that each of the three classes of Chen modules have different $A$-dimension vectors, and thus these classes are pairwise inequivalent to each other.

\subsubsection{Connection with representations of the Cuntz algebra.}
Chen's construction was inspired by a certain family of representations of the Cuntz algebra $\cO_n$ constructed in \cite{Lawrynowicz-Nono-Suzuki03} and equivalent to the \textit{permutation} representations of \cite{Bratteli-Jorgensen99}.
These representations of $\cO_n$ are analogous to the untwisted Chen modules $F_p^1$ and $F_p$.
The twisted cases $F_p^\lambda$ correspond to the larger class of \emph{atomic} representations $\cO_n$ first constructed in \cite{DHJ-atomic15}.
They were deeply studied from the viewpoint of Pythagorean representations in a previous paper of the present authors \cite{Brothier-Wijesena24a}.
We showed that their restriction to Thompson's group $V$ are monomial representations associated to parabolic subgroups.
Similar results can be adapted to Chen modules over the Leavitt algebra $L$ (when $\Ga$ is a bouquet of $n$ loops) and their restriction to the $n$-ary version of Thompson's group $V$.

\subsection{Other Chen modules.}

The construction of Chen was subsequently generalised by others to obtain new families of simple $L$-modules which are also termed as Chen modules. We shall continue to use the same notation from the preceding subsection.

\subsubsection{Chen modules from twisting by irreducible polynomial.}
These simple modules were constructed by Ara and Rangaswamy \cite{Ara-Rangaswamy14}. 
It is assumed $\bk$ is not algebraically closed and $\Gamma$ has an \emph{exclusive} cyclic $c$ - every vertex on $c$ does not belong to another cyclic.

Fix an irreducible monic polynomial $P\in\bk[x]$ giving the field extension $\bk[x]/P$ generated by $\bk$ and $x$. Let $q=c^\infty$ be a periodic ray with $c$ an exclusive cycle.
Construct the usual Chen module $F^{x}_q$ for the Leavitt path algebra $\ti L$ over $\bk[x]/P$ and make it an $L$-module, denoted $F_q^P$, using the obvious embedding $L\into \ti L$.
When $P(x) = x - \lambda, \lambda\in\bk^\times$, then $F^P_q$ is the Chen module $F^{\lambda}_q$ described in the previous subsection.
The authors showed that each of these modules are simple, are classified by $(P,[q])$, and they are not equivalent to any of the other Chen modules.

We now describe a simple $A$-module $C_q^P$ living inside $F^P_q$ so that $\Pi(C_q^P)\simeq F^P_q.$
Let $e$ be the first edge in the cycle $c$. Then define $C_q^P$ to be the $A$-module with space
\[(\bk[x]/P) \otimes \bk[T_q]\]
(the tensor product being over $\bk$) and $A$-module structure
\begin{equation*}
	(h \otimes r) \cdot f := 
	\begin{cases*}
		h \otimes f^*r \textrm{, }&$\quad f \neq e$\\
		xh \otimes f^*r \textrm{, }&$\quad f = e$.
	\end{cases*} \text{ for } h\in \bk[x]/P \ ,\ r\in T_q \ , \ f\in E^1.
\end{equation*}
It is not difficult to identify $C_q^P$ with an $A$-submodule of $\Lambda(F_q^P)$ that is complete.
Having that $\bk[x]/P$ is a field easily implies that all vectors are cyclic and thus $C_q^P$ is simple.

The construction of the modules $F_q^P$ were generalised for \emph{arbitrary} cycles by \'{A}nh and Nam in \cite{Anh-Nam21} where they proved similar classification results. 
Our construction of $C_q^P$ does not use the fact that $c$ is exclusive and hence permits to obtain these latter modules.

\subsubsection{Chen modules from infinite emitters.}
The Chen module $N_\omega$ from the previous subsection was generalised into $S_{\omega\infty}$ for $\omega$ an infinite emitter as being the free vector space over the set of paths ending in $\omega$ \cite{Ara-Rangaswamy14, Rangaswamy15}.
The authors showed that these modules are all simple and are not isomorphic to any of the above Chen modules. 
Note that these constructions cannot be obtained with our techniques.

\subsubsection{Further generalisations of Chen modules.}
We extend the $L$-modules $F_q^P$ using a non-algebraically closed field $\bk$. 
Take a vector space $K$, a linear map $X:K\to K$ with no nontrivial proper invariant subspaces, and a periodic ray $q=c^\infty$.
Define the $A$-module $C_q^X:=K\otimes \bk[T_q]$ as $C_q^P$ but where $\bk[x]/P$ is replaced by $K$ and where the multiplication by $x$ is replaced by the map $X$.
It is easy to see that $C_q^X$ is irreducible and are classified by $[q]$ and the pair $(K,X)$ up to equivariant isomorphisms.

Here is another generalisation.
Consider a cyclic ray $q = c^\infty$ and let $r$ be the ray such that ${}_1r = q$. 
Fix a nonzero vector $v \in \bk[T_q]$.
Then define an $A$-module $C_q^v := \bk[T_q]$ with action given by 
$$p \cdot e := e^*p \textrm{ for } p \neq r \textrm{ and } r \cdot e := v, \textrm{ for all } p \in T_q,\ e \in E^1.$$
Certain choices of $v$ yield simple $L$-modules $F_q^v$ that are inequivalent to any of the above Chen modules (though not all choices of $v$ will give simple modules).

\subsection{Graded-representations} \label{subsec:graded-irrep}
The Leavitt path algebra $L$ has a natural $\Z$-grading $\vert \cdot \vert$ given by 
$$\vert pq^*\vert := \vert p \vert - \vert q \vert$$
where $\vert p \vert$ denotes the length of the path $p$. Subsequently, it is natural to study \textit{graded} representations of $L$. A representation $W$ of $L$ is graded if $W$ is a graded vector space, so that $W = \oplus_{j \in \Z} W_j$, and $W_j \cdot l \subset W_{j + \vert l \vert}$ for all homogeneous elements $l$ in $L$. 
Graded-subrepresentations are defined similarly and a representation is \textit{graded-irreducible} if there are no non-trivial proper graded-subrepresentations. 
Not all representations of $L$ are graded. For example, the Chen representation $F_p^\lambda$ where $p = c^\infty$ is rational cannot be endowed with such a $\Z$-grading structure (this is because $p \cdot c = \lambda p$ and $\vert c \vert > 0$). 
Irreducibility implies graded-irreducibility but the converse does not hold in general.

\subsubsection{Graded-irreducible representations.}
We now present modules appearing in \cite{Hazrat-Rangaswamy16,Vas23} using an alternative definition.
Fix a graph $\Gamma$ containing a cycle $c$ without any exits - there is no edge which does not lie on $c$ but its source vertex is in $c$. Let $\nu$ be the vertex based on $c$ so that $c = c_1\dots c_n$ and $s(c_1) = \nu$. Define $N_{\nu c}$ to be the $\bk$-vector space spanned by the set:
\[B = \{pq^* \in L : p,q \in E^*,\ r(q^*) = \nu\}.\]
Note, since $c$ does not have any exits the path $q$ must be lying in $c$ and $(pq^*)^*(pq^*) = \nu$.
An $L$-action is defined on $N_{\nu c}$ similarly to the Chen representations:
\[pq^*\cdot e := e^*pq^*,\ pq^*\cdot e^* := epq^*.\]
There is a natural grading structure $\deg$ on $N_{\nu c}$ given by $\deg(pq^*) := \deg(p) - \deg(q)$ for $pq^* \in N_{\nu c}$. Since this is compatible with the grading structure on $L$ it is immediate that $N_{\nu c}$ is a graded representation of $L$. 

The $N_{\nu c}$ are reducible but are graded-irreducible.
To show that $N_{\nu c}$ is graded-irreducible (we will show that it is not irreducible below) consider any homogeneous vector $pq^* \in N_{\nu c}$. Acting by $pq^*$ on the right we can assume the vector is $\nu$. However, $\nu$ is a cyclic vector in $N_{\nu c}$ since $\nu \cdot (qp^*) = pq^* \in B$ for any $p,q \in E^*$ with $r(q^*) = \nu$. This shows that $N_{\nu c}$ does not have any non-trivial proper graded-subrepresentations.

\subsubsection{$N_{\nu c}$ arising from $A$-modules.}
Define $G_{\nu c}$ to be the $\bk$-vector space spanned by
\[B := \{q^* : q \in E^*,\ r(q^*) = \nu\}\]
which forms a complete submodule of $\Lambda(N_{\nu c})$. Observe there is a natural $\N$ grading $\vert \cdot \vert_A$ on $G_{\nu c}$ which lifts to the $\Z$ grading on $N_{\nu c}$.
 
Consider the vector $\nu + c^* \in G_{\nu c}$. For any path $p$ of length greater than $0$ and $s(p) = \nu$, we have that $\vert \nu \cdot p \vert_A > \vert \nu \vert_A$ and $\vert c^* \cdot p \vert_A > \vert c^* \vert_A$. Hence, we deduce that $\nu$ is not in the $A$-submodule generated by $\nu+c^*$, and this submodule is not complete (e.g. consider $\nu \in G_{\nu c}$).
Hence, by our results this reproves $N_{\nu c}$ is reducible.
Moreover, define the set $B_k := \{r(c^*)^k : r \in B\} \subset B$ and vector spaces $G_{\nu c}^k := \bk[B_k]$ for $k \in \N$. These form a descending chain of complete $A$-submodules of $\Lambda(N_{\nu c})$ with trivial intersection:
\[G_{\nu c} = G_{\nu c}^0 \supset G_{\nu c}^1 \supset G_{\nu c}^2 \supset \dots\]
which shows that $N_{\nu c} \notin \cS$.

\subsubsection{Generalising $N_{vc}$ for arbitrary cycles.}
Taking inspiration from the above construction, 
and using the techniques we have developed in this paper, we extend $N_{\nu c}$ to be defined for \emph{arbitrary} cycles (which may contain exits) to obtain a large family of representations of $L$ that are graded-irreducible but not irreducible.

Let $\Gamma$ be any graph containing a cycle $c$ based at the vertex $\nu$. Define $G_{\nu c}$ to be the $\bk$-vector space spanned by the below set:
\[\{q^* : q \textrm { is a path lying in the cycle } c,\ r(q^*) = \nu\}.\]
Define an $A$-action on $G_{\nu c}$ by $q^* \cdot e = e^*q^*$ if $e$ is in $c$, otherwise $q^* \cdot e = 0$. 
A similar argument as above shows the resulting $L$-module $\Pi(G_{\nu c})$ is graded-irreducible but reducible.

Similarly to the Chen modules of \cite{Chen15}, the modules $N_{\nu c}$ correspond to certain atomic representations of $\cO_n$ in \cite{DHJ-atomic15} which are also reducible. The restriction of these atomic representations to Thompson's group $V$ were not studied in \cite{Brothier-Wijesena24a}. However, it can be shown that their restrictions are equivalent to monomial representations associated to the subgroups $\widehat V_p := \{g \in V: g(p) = p,\ g'(p) = 1\}$.

\subsection{Simple (generalised) Chen modules in the base case}\label{sec:Chen-module-few}
Fix $n\geq 2$ and assume that $\Ga$ is a bouquet of $n$ loops.
We are interested in simple Chen modules of $L$ and their generalisations.
The graph $\Ga$ has no sinks hence we do not have any $N_\omega$.
The Chen modules that remains are $F_q^\lambda$ with $q=c^\infty, \lambda\in \bk^\times$ and $F_p$ with $p$ irrational.
We have $\dim_A(F_q^\lambda)=|c|$ and $\dim_A(F_p)=\varepsilon$ ($F_p\notin\cS$).
For the former case, $c$ is a prime word in the edges $e_1,\dots,e_n$ and $F_q^\lambda\simeq F_r^\mu$ if and only if $\lambda=\mu$ and if $r=c_0^\infty$ where $c_0$ is equal to $c$ up to rotation.
Write $W_d$ for the (finite) set of prime words of length $d$ mod out by rotations.
We deduce that the classes of Chen modules of $A$-dimension $d$ is parameterised by $W_d\times \bk^\times$ and form a one-dimensional sub-variety of $\Irr(L)_d$.

Let us turn to generalisation of Chen modules.
Since $\Ga$ is a bouquet of loops it has no exclusive cycles and no infinite emitters, and thus no modules $S_{\omega\infty}$, $F^P_q$ from the works of Ara and Rangaswamy.
Though, we do have the generalisations $F^P_q$ for $q = c^\infty$ from the work of Ánh and Nam \cite{Anh-Nam21} (when $\bk$ is not algebraically closed). Let $\mathbb{IP}_{b}$ denote the set of monic irreducible polynomials (over $\bk$) of degree $b$. The classes of these modules of $A$-dimension $db$ is parameterised by $W_d \times \mathbb{IP}_{b}$. These coincide with the Chen modules when $b=1$.
Finally, the graded-modules of \cite{Hazrat-Rangaswamy16, Vas23} are reducible.

To summarise: simple generalised Chen modules are the $F_p$, $F_q^\lambda$ (with $q=c^\infty$) and the $F_q^P$ of \'Anh--Nam having $A$-dimension $\varepsilon,|c|,|c|\cdot \deg(P)$, respectively.
Fix $d\geq 1$ finite and consider those with $A$-dimension $d$.
The irreducible classes $\Irr(L)_d$ forms a smooth variety of dimension $(n-1)d^2+1$ when $\bk$ is algebraically closed of characteristic one.
The Chen modules $F_q^\lambda$ with $q=c^\infty, |c|=d$ form a subvariety of dimension $1$ in that case and there are no other generalised Chen modules of this dimension.
Now, if $\bk$ is not algebraically closed we additionally have some \'Anh--Nam modules $F_q^P.$
For each decomposition $d=ab$ into natural numbers we have a family of modules $F^P_q$ with $q=c^\infty, |c|=a$ and $P$ monic irreducible with degree $b$.
These classes are all parameterised by at most $d$ coefficients while $\Irr(L)_d$ remains parameterised by $(n-1)d^2+1$ coefficients.

Even in $A$-dimension equal to $1$ we get a profusion of new simple $L$-modules from our techniques. Indeed, any vector $v\in \bk^n$ with at least two nonzero entries defines an $A$-module $V$ so that $W:=\Pi(V)$ is a simple $L$-module that is easily described by decorated trees and is not isomorphic to any previously constructed $L$-module.

\section{Comparison with the work of Ko\c{c} and Özaydın} 
In this section we continue to assume $\Gamma$ is row-finite. 
As before $A,L$ are the associated quiver algebra and Leavitt path algebra. 

\subsection{Work of Ko\c{c} and Özaydın}
Ko\c{c} and Özaydın in \cite{Koc-Ozaydin20} made the key observation that $\Rep(L)$ is equivalent to the full subcategory of $\Rep(A)$ consisting of $A$-modules satisfying the so-called \textit{isomorphism condition}:
\[\bigoplus_{e \in E_\nu^1}\pi(e) : V_\nu \rightarrow \bigoplus_{e\in E_\nu^1} V_{re} \textrm{ is an isomorphism} \tag{I}.\]
This condition means that each caret map is an isomorphism.
The equivalence of Ko\c{c}-Özaydın is obtained via the same forgetful functor $\Lambda : \Rep(L) \rightarrow \Rep(A)$ as we have defined in Section \ref{sec:path-algebra}. Hence, the isomorphism condition (I) precisely describes the $A$-modules in the image of $\Lambda$. Subsequently, given an $A$-module $V$ satisfying (I) one obtains an $L$-module structure on the same vectors space $V$ (which will be equivalent to $\Pi(V)$). 

The authors used the above correspondence to re-construct the Chen modules $F^1_q, F_p, N_v$ of \cite{Chen15} using $A$-modules (see the examples in \cite[Section 5]{Koc-Ozaydin20}). 
Then they obtained some classification results in \cite{Koc-Ozaydin18} and \cite{Koc-Ozaydin22}: in the first for finite-dimensional $L$-modules, and in the second they classified all irreducible $L$-modules (showing they are all Chen modules or generalisations of them) for specific $L$.

Returning back to \cite{Koc-Ozaydin20}, while the authors initially consider $A$-modules satisfying condition (I), they also studied the functor $\Psi : \Rep(A) \rightarrow \Rep(L)$ given by
\[V \mapsto V \otimes_{A} L,\ f \mapsto f \otimes_A \id_L.\]
This functor is in fact naturally isomorphic to the functor $\Pi$ in the present paper. The objects $\Psi(V)$, $\Pi(V)$ are isomorphic by:
\[\theta: \Psi(V) \rightarrow \Pi(V),\ x \otimes_A p^* \mapsto [p,x] \textrm{ for all } x \in V,\ p \in E^*.\] 
The authors also showed how $\Psi(V)$ can be constructed as a direct limit which is similar to the construction of $\Pi(V)$ in the present paper. The authors then considered the full subcategory $\Rep_{\ker}(A)$ of $\Rep(A)$ consisting of all objects $V$ such that $\Psi(V) \simeq \{0\}$ (note, this is stronger than being degenerate). They showed that $\Rep_{\ker}(A)$ is a Serre subcategory of $\Rep(A)$ and thus one can consider the quotient category $\Rep(A)/\Rep_{\ker}(A)$. They then concluded the paper by showing this quotient category is equivalent to $\Rep(L)$. 

\subsection{Comparison to present paper}
Ko\c{c}--Özaydın's realised an equivalence between $\Rep(A)/\Rep_{\ker}(A)$ and $\Rep(L)$ via an explicit construction. 
This is a beautiful theoretical result that they mainly use to classify $L$-modules $W$ in the cases:
\begin{itemize}
    \item $W$ is finite-dimensional; or
    \item $W$ is irreducible but $L$ has few representations (the main examples being $L$ with Gelfand--Kirillov dimension smaller than 4, see \cite{Gelfand-Kirillov66}). 
\end{itemize}
In the latter case $W$ is always a (generalised) Chen module.

Our approach is relevant for all row-finite graphs and focus on $A$-modules with the introduction of the key properties of fullness,  degeneracy and the notion of $A$-dimension.
We provide an explicit equivalence of categories between full subcategories of $\Rep(A)$ and $\Rep(L)$ giving bijection of hom-spaces and moreover preserving irreducibility and indecomposability. This permits to reduce the study of representations of $L$ to the study of representations of $A$ in an effective manner.
Our base example is the Leavitt algebra ($\Ga$ is a bouquet of loops)\,---\, that has no finite-dimensional $L$-modules and has infinite Gelfand--Kirilov dimension. 
It has many irreducible classes of $L$-modules of finite $A$-dimensions, and the Chen modules (and their generalisation) form a tiny piece of their representation theory, see Section \ref{sec:Chen-module-few}. 

\newcommand{\etalchar}[1]{$^{#1}$}

\end{document}